\documentclass[UTF8,11pt]{article}
\usepackage{amssymb,amsmath,color,amsfonts,float,amscd,amsthm,bm,mathrsfs,array,tikz,dutchcal}
\usepackage[title]{appendix}

\newcommand*{\circled}[1]{\lower.7ex\hbox{\tikz\draw (0pt, 0pt)%
    circle (.5em) node {\makebox[1em][c]{\small #1}};}}
 \allowdisplaybreaks

\catcode`!=11
\let\!int\int \def\int{\displaystyle\!int}
\let\!lim\lim \def\lim{\displaystyle\!lim}
\let\!sum\sum \def\sum{\displaystyle\!sum}
\let\!sup\sup \def\sup{\displaystyle\!sup}
\let\!inf\inf \def\inf{\displaystyle\!inf}
\let\!cap\cap \def\cap{\displaystyle\!cap}
\let\!max\max \def\max{\displaystyle\!max}
\let\!min\min \def\min{\displaystyle\!min}

\catcode`!=12

\let\oldsection\section
\renewcommand\section{\setcounter{equation}{0}\oldsection}

\allowdisplaybreaks

\newtheorem{defn}{Definition}[section]
\newtheorem{thm}{Theorem}[section]
\newtheorem{prop}[thm]{Proposition}
\newtheorem{lem}[thm]{Lemma}

\newtheorem{rem}[thm]{Remark}

\begin{document}

\title{Remarks on Sharp Interface Limit for an Incompressible Navier-Stokes and Allen-Cahn Coupled System}
\author{Song Jiang\footnote{Institute of Applied Physics and Computational Mathematics. Email: jiang@iapcm.ac.cn}
\and
Xiangxiang Su\footnote{School of Mathematical Sciences, Shanghai Jiao Tong University, Shanghai 200240, P. R. China. Email: sjtusxx@sjtu.edu.cn}
\and
Feng Xie
\footnote{Corresponding author.
School of Mathematical Sciences, and CMA-Shanghai, Shanghai Jiao Tong University, Shanghai 200240, P. R. China.
Email: tzxief@sjtu.edu.cn}
}
\date{}
\maketitle

\noindent{\bf Abstract:} We are concerned with the sharp interface limit for an incompressible Navier-Stokes and Allen-Cahn coupled system
in this paper. When the thickness of the diffuse interfacial zone, which is parameterized by $\varepsilon$, goes to zero,
we prove that a solution of the incompressible Navier-Stokes and Allen-Cahn coupled system converges to a solution
of a sharp interface model in the $L^\infty(L^2)\cap L^2(H^1)$ sense on a uniform time interval independent of
the small parameter $\varepsilon$. The proof consists of two parts: one is the construction of a suitable approximate solution
and another is the estimate of the error functions in Sobolev spaces. Besides the careful energy estimates,
a spectral estimate of the linearized operator for the incompressible Navier-Stokes and Allen-Cahn coupled system
around the approximate solution is essentially used to derive the uniform estimates of the error functions.
The convergence of the velocity is well expected due to the fact that the layer of the velocity across the diffuse interfacial
zone is relatively weak.
\vskip0.2cm


%

\section{Introduction and Main Results}

\hspace{2em}
The two-phase flow finds many applications in chemistry and engineering sciences. It also produces many interesting but challenging mathematical problems from both analysis
and numerical simulation points of view.
Basically, there are two widely used models: the sharp interface model and the diffuse interface model respectively.
The sharp interface model is related to a free boundary value problem. That is, the two fluids
are separated by an interface $\Gamma$, where the interface $\Gamma$ is a lower dimensional surface, which will be determined together with the motion of two fluids. In general, such a sharp interface model
is hard to be handled in the numerical simulation. Thus, the so-called diffuse interface model (also known as the phase field model) is introduced accordingly, where the sharp interface is replaced
by an interfacial region, which takes into account that the two fluids have been mixing to a certain extent in the interfacial region. Here the width of diffuse interfacial zone is parameterized
by a small parameter $\varepsilon$. And an order parameter, which will be represented by $c_\varepsilon$, is also introduced. It  takes two different values (for example, +1 and -1) in each phase,
and changes smoothly between the two values in the diffuse interfacial zone. The basic diffuse interface model of a two-phase flow for two macroscopically immiscible viscous Newtonian fluids with
the same density can be traced back to Hohenberg and Halperin \cite{Hohenberg}, which is named as ``Model H". Such a model is described by the incompressible Navier-Stokes/Cahn-Hilliard coupled system in \cite{Gurtin}.

Under suitable initial and boundary conditions, this diffuse interface model for two-phase flows of incompressible fluids was shown to admit both weak and strong solutions
in 2D and 3D bounded domains in \cite{Abels 0,Abels 2,Gal,Gal2}. And the asymptotic stability of solutions to the diffuse interface model was given in \cite{Boyer}.
We also refer to \cite{Gal}, where the authors established the existence of the exponential attractor, and obtained at the same time the estimates of the convergence rate in the phase-space metric.

The corresponding sharp interface model has also been extensively studied. The local in time existence of strong solutions was established in \cite{Abels},
and the long time existence of weak solutions was shown in \cite{Abels 3}. One can refer to \cite{Kohne,Pruss0,Pruss} and the references cited therein for the related results in this field.

There are many extensively studies about the Cahn-Hilliard equation and the Allen-Cahn equation respectively. For the fourth-order model, the global existence and the time decay estimates
of smooth solutions in the $L^p$ sense to the Cauchy problem were established in \cite{Liu-W-Z}.
The sharp interface limit for the Cahn-Hilliard equation was considered in \cite{Alikakos 0} by the method of matched asymptotical expansions, while the Navier-Stokes/Cahn-Hilliard coupled system
was analyzed through the Fourier-spectral method for the numerical approximation in \cite{Liu}.

From the work \cite{Elliott}, we see that it is naturally related to the Allen-Cahn equation when some complex geometric problems are considered (``motion by mean curvature'').
  There have been a lot of references on this second-order equation in both one- and multi-dimensional cases, see \cite{Chen X0,de Mottoni,Evans,Fusco} for example.
  Also, the axisymmetric solutions of the Navier-Stokes/Allen-Cahn system in $\mathbb{R}^3$ was investigated in \cite{Xu}, where the authors proved the global regularity of solutions for both large viscosity and small initial data cases.

  The vanishing viscosity limit of the Navier-Stokes/Allen-Cahn system was studied in \cite{Zhao}, where the authors proved that a global weak solution of the Navier-Stokes/Allen-Cahn system converges in the $L^2$ sense to the locally smooth solution of the Euler/Allen-Cahn system on a small time interval.

Besides the aforementioned results on sharp and diffuse interface models of two-phase flows,
it is interesting from both mathematical and application points of view to study
the sharp interface limit from a diffuse interface model to a sharp interface one.
Recently, Abels and Liu \cite{Liu Y N} proved that a weak solution of the Stokes/Allen-Cahn system converges to the solution of
a sharp interface model over a small time interval, also cf. \cite{Liu-Y}. The sharp interface limit for the Navier-Stokes/Allen-Cahn system
was studied more recently by Abels and Fei in \cite{Fei}, while the sharp interface limit for the Stokes/Cahn-Hilliard coupled system
was dealt with in \cite{Abels 4}.
 We point out that this paper is concerned with the same sharp interface limit problem for the Navier-Stokes/Allen-Cahn system
as studied in \cite{Fei}. However, the approaches used to derive the estimates of the error functions between
this paper and \cite{Fei} are different. In particular, it is diverse in estimating the derivatives of the error functions,
which will be discussed in details later.

Precisely, we are concerned with the sharp interface limit of solution to an incompressible Navier-Stokes/Allen-Cahn coupled
system in a bounded domain $\Omega \subset \mathbb{R}^{2}$:
\begin{align}
\partial_t\mathbf{v}_\varepsilon+\mathbf{v}_\varepsilon \nabla \mathbf{v}_\varepsilon-\Delta \mathbf{v}_{\varepsilon}+\nabla p_{\varepsilon} &=-\varepsilon \operatorname{div} (\nabla c_{\varepsilon} \otimes \nabla c_{\varepsilon}) && \text { in } \Omega \times(0, T_{1}), \label{9.6.4}\\
\operatorname{div} \mathbf{v}_{\varepsilon} &=0 && \text { in } \Omega \times (0, T_{1}), \\
\partial_{t} c_{\varepsilon}+\mathbf{v}_{\varepsilon} \cdot \nabla c_{\varepsilon} &=\Delta c_{\varepsilon}-\frac{1}{\varepsilon^{2}} f^{\prime}\left(c_{\varepsilon}\right) && \text { in } \Omega \times (0, T_{1}), \label{10.19.1}\\
\mathbf{v}_{\varepsilon}|_{\partial \Omega} = 0, \qquad &  c_{\varepsilon}|_{\partial \Omega}=-1 &&\text { on } \partial \Omega \times (0, T_{1}), \\
\mathbf{v}_{\varepsilon}|_{t=0} =\mathbf{v}_{0, \varepsilon}, \qquad & c_{\varepsilon}|_{t=0} =c_{0, \varepsilon} && \text { in } \Omega, \label{9.6.5}
\end{align}
 where $\mathbf{v}_{\varepsilon}$ stands for the velocity vector, $ p_{\varepsilon}$ denotes the pressure, $c_{\varepsilon}$ is the order parameter related to the fluid concentration (for example, the concentration difference or the concentration of one component), and $\varepsilon$
is a small positive parameter which describes the ``thickness" of the diffuse interfacial region. As in \cite{Chen X1, Chen X}, the potential function $f$ satisfies
\begin{eqnarray}
\label{PF}
f \in C^\infty(\mathbb{R}),\quad f'(\pm1)=0,\quad f''(\pm1)> 0,\quad f(c)=f(-c)> 0 \quad \forall c \in(-1,1).
\end{eqnarray}
A typical example is $f(c)=\frac{1}{8}(1-c^2)^2$, which is also the potential function considered in this paper. We believe that the main results in this paper can be extended to the general potential function case (\ref{PF}) without any essential difficulties.

Multiplying (\ref{10.19.1}) by $\varepsilon^2\Delta c_{\varepsilon}-f^{\prime}(c_{\varepsilon})$ and integrating by parts, one obtains
\begin{align*}
\int_{\Omega} (\varepsilon \Delta c_{\varepsilon}-\frac{1}{\varepsilon} f^{\prime}(c_{\varepsilon}))^{2} \mathrm{d} x=& \int_{\Omega} (\varepsilon \Delta c_{\varepsilon}-\frac{1}{\varepsilon} f^{\prime}(c_{\varepsilon})) (\varepsilon\partial_{t} c_{\varepsilon}+\varepsilon\mathbf{v}_{\varepsilon} \cdot \nabla c_{\varepsilon} )\,\mathrm{d} x\\
=&\int_{\Omega} -\partial_{t}\left[\frac{\varepsilon^2}{2}|\nabla c_{\varepsilon}|^2+f(c_{\varepsilon})\right]-\frac{\varepsilon^2}{2} \nabla \mathbf{v}_{\varepsilon} \nabla c_{\varepsilon}\otimes \nabla c_{\varepsilon}\, \mathrm{d} x,
\end{align*}
where the divergence-free condition of $\mathbf{v}_{\varepsilon}$ is used in the last equality. Moreover, it follows from (\ref{9.6.4}) that
\begin{align*}
\int_{\Omega} \varepsilon \nabla \mathbf{v}_{\varepsilon} \nabla c_{\varepsilon}\otimes \nabla c_{\varepsilon}\, \mathrm{d} x=\int_{\Omega}
\Big( \frac{1}{2}\partial_{t}(\mathbf{v}_{\varepsilon})^2+(\nabla \mathbf{v}_{\varepsilon})^2 \Big)\mathrm{d} x.
\end{align*}
Thus, putting the above two equalities together and integrating the resulting equality with respect to $t$, we arrive at the basic
energy equality for solutions to the equations (\ref{9.6.4})-(\ref{9.6.5}):
\begin{eqnarray} \label{10.18.3}
E_{\varepsilon}^{tot}\left(c_{\varepsilon}(t)\right)+\int_{0}^{t} \int_{\Omega}(\frac{1}{2}|\nabla \mathbf{v}_{\varepsilon}|^{2}+\frac{1}{\varepsilon}|\mu_{\varepsilon}|^{2}) \mathrm{d} x\,\mathrm{d} \varsigma=E_{\varepsilon}^{tot}(c_{0, \varepsilon}) \quad\mbox{ for all } t \in (0, T_{1}),
\end{eqnarray}
 where
\begin{align} \label{3.18.2}
\mu_{\varepsilon}=-\varepsilon \Delta c_{\varepsilon}+\frac{1}{\varepsilon} f^{\prime}(c_{\varepsilon}),
 \end{align}
and
$$
E_{\varepsilon}\left(c_{\varepsilon}(t)\right)=\int_{\Omega} \varepsilon \frac{|\nabla c_{\varepsilon}(x, t)|^{2}}{2} +\frac{f\left(c_{\varepsilon}(x, t)\right)}{\varepsilon} \mathrm{d} x,\quad E_{\varepsilon}^{tot}\left(c_{\varepsilon}(t)\right)=\int_{\Omega} \frac{|\mathbf{v}_{\varepsilon}(x, t)|^{2}}{4} \mathrm{~d} x+E_{\varepsilon}\left(c_{\varepsilon}(t)\right).
$$

From the energy equality (\ref{10.18.3}), we find that $\mathbf{v}_{\varepsilon}$ has no strong layer across the diffuse interfacial zone as $\varepsilon$ goes to zero. This fact will be used in the construction of approximate solution of $\mathbf{v}_{\varepsilon}$.
One can refer to Section 3 for the details.

When the thickness parameter $\varepsilon$ in \eqref{9.6.4}-\eqref{9.6.5} goes to zero, the diffuse interfacial
zone will shrink into a lower dimensional surface $\Gamma_t\subseteq\Omega$ (which excludes contact angle problems),
which is a free boundary. And then $\Omega$ is separated
into two smooth domains $\Omega^\pm(t)$ by the sharp interface $\Gamma_t$ for each $t\in[0, T_0]$, where $\Omega^+$ is the internal
domain and $\Omega^-$ is the external domain. And the order parameter takes values of $-1$ in $\Omega^{-}$ and $1$ in $\Omega^{+}$
respectively in the limit case. The velocity $\mathbf{v}$ is expected to continuous across the sharp interface $\Gamma_t$
due to the diffusion effect of the velocity. Moreover, it also satisfies a surface tensor constrain.
Consequently, the sharp interface
limit problem for (\ref{9.6.4})-(\ref{9.6.5}), we shall prove, is the following free boundary value problem:
\begin{align}
\partial_t\mathbf{v}+\mathbf{v} \nabla \mathbf{v}-\Delta \mathbf{v}+\nabla p & =0 && \text { in } \Omega^{\pm}(t),\  t \in [0, T_{0}], \label{11.8.1}\\
\operatorname{div} \mathbf{v} & =0 && \text { in } \Omega^{\pm}(t),\  t \in [0, T_{0}], \\
{[2 D \mathbf{v}-p \mathbf{I}] \mathbf{n}_{\Gamma_{t}}} & =-\sigma H_{\Gamma_{t}} \mathbf{n}_{\Gamma_{t}} && \text { on } \Gamma_{t},\  t \in [0, T_{0}], \\
{[\mathbf{v}]} & =0 && \text { on } \Gamma_{t},\  t \in [0, T_{0}], \\
\mathbf{v}|_{\partial \Omega} & =0 && \text { on } \partial \Omega \times [0, T_{0}], \\
V_{\Gamma_{t}}-\mathbf{n}_{\Gamma_{t}} \cdot \mathbf{v}|_{\Gamma_{t}} & =H_{\Gamma_{t}} && \text { on } \Gamma_{t},\  t \in [0, T_{0}] \label{11.8.2},
\end{align}
where $D \mathbf{v}=\frac{1}{2}(\nabla\mathbf{v}+(\nabla\mathbf{v})^T)$ is the symmetric part of the velocity gradient tensor. $\mathbf{n}_{\Gamma_t}$ is the unit interior normal of $\Gamma_t$ with respect to $\Omega^{+}(t)$, and
$$
[h](p, t)=\lim _{d \rightarrow 0+}[h\left(p+\mathbf{n}_{\Gamma_{t}}(p) d\right)-h\left(p-\mathbf{n}_{\Gamma_{t}}(p) d\right)].
$$
And $H_{\Gamma_{t}}$ and $V_{\Gamma_{t}}$ are the curvature and the normal velocity of the interface $\Gamma_t$,
respectively. $\sigma$ is the coefficient of surface tension.

To determine $\sigma$, it is necessary to introduce the following profile $\theta_0$, which is the unique
increasing solution of
$$
-\theta''_0(\rho)+f'(\theta_0(\rho))=0,
$$
together with $\theta_0(0)=0$ and $\theta_0(\rho)\to\pm 1$ as $\rho\to\pm\infty$.

According to the properties on the above second order ordinary differential equation, we know that $\theta_0$ also satisfies
\begin{eqnarray}
|\partial^m_\rho(\theta_0(\rho)\mp1)|=O(e^{-\alpha|\rho|}) \quad \text{for } \text{ all } \rho \in \mathbb{R},
\end{eqnarray}
where $\alpha \text{ = min}(\sqrt{f''(-1)},\sqrt{f''(1)})$. Then, the coefficient $\sigma$ is given by
$\sigma=\int_{\mathbb{R}} \theta'_0(\rho)^2 \mathrm{d}\rho$.

To state the main theorem, it is helpful to introduce the approximate solution constructed in this paper roughly here, which will be constructed by using the two-scale matched asymptotical expansion method in Subsection 3.1. The approximate solution, denoted by $(\mathbf{{\tilde v}}_A, c_A)$, will act as a bridge between the solutions to (\ref{9.6.4})-(\ref{9.6.5}) and the solutions to (\ref{11.8.1})-(\ref{11.8.2}). To this end, we introduce the following notations. For $\delta>0$, and $t\in [0, T_0]$, we define the tubular neighborhoods of $\Gamma_t$,
\begin{align*}
\Gamma_t(\delta)\triangleq \{y\in \Omega: \hbox{dist}(y, \Gamma_t)<\delta\},\quad \Gamma(\delta)=\bigcup_{t\in [0, T_0]}\Gamma_t(\delta)\times\{t\},
\end{align*}
and the signed distance function
$$
d_{\Gamma}(x, t)\triangleq \operatorname{sdist}(\Gamma_{t}, x)=
\begin{cases}
\operatorname{dist}\left(\Omega^{-}(t), x\right) & \text { if } x \notin \Omega^{-}(t), \\
-\operatorname{dist}\left(\Omega^{+}(t), x\right) & \text { if } x \in \Omega^{-}(t).\end{cases}
$$
Let $\zeta(s)\in C^\infty(\mathbb{R})$ be a cut-off function, which is defined as follows.
\begin{align*}
\zeta(s)=1\;\hbox{ for }\; |s|\leq\delta;\quad \zeta(s)=0\;\hbox{ for }\; |s|>2\delta;\quad 0\leq -s\zeta'(s)
\leq 4\;\hbox{ for }\;\delta\leq |s|\leq 2\delta.
\end{align*}
Then the approximate solution $(\mathbf{\tilde{v}}_A, c_{A})$ constructed in this paper takes the following form, also refer to \cite{Liu Y N}.
\begin{eqnarray*}
\begin{split}
\mathbf{v}_A^{in}(\rho,x,t)&=\mathbf{v}_0(\rho,x,t)+\varepsilon \mathbf{v}_1(\rho,x,t)+\varepsilon^2 \mathbf{v}_2(\rho,x,t),\;\; \mathbf{v}_A^{\pm}(x,t)=\mathbf{v}_0^{\pm}(x,t)+\varepsilon \mathbf{v}_1^{\pm}(x,t)+\varepsilon^2 \mathbf{v}_2^{\pm}(x,t),\\
\mathbf{v}_A(x,t)&=\zeta \circ d_\Gamma \mathbf{v}_A^{in}(\rho,x,t)+(1-\zeta \circ d_\Gamma)(\mathbf{v}_A^+(x,t)\chi_+ +\mathbf{v}_A^-(x,t)\chi_-),\quad \mathbf{\tilde{v}}_A=\mathbf{v}_A+\varepsilon^2 \,\mathbf{f},\\
c^{i n}(x, t)&= c_{0}^{i n}(x, t)+\varepsilon^{2} c_{2}^{i n}(x, t)+\varepsilon^{3} c_{3}^{i n}(x, t),\;\;c_{A}(x, t) =\zeta \circ d_{\Gamma} c^{i n}(x, t)+(1-\zeta \circ d_{\Gamma})(\chi_{+}-\chi_{-}),
\end{split}
\end{eqnarray*}
with the stretched variable
\begin{align*}
\rho=\frac{d_{\Gamma(x, t)}}{\varepsilon}-h_1(S(x,t), t)-\varepsilon h_{2,\varepsilon}(S(x, t), t),
\end{align*}
where $h_1(S(x,t), t)$ and $h_{2,\varepsilon}(S(x, t), t)$ are two functions to be determined later.

Now, we state the main theorem in this paper.
\begin{thm}   \label{theorem:thm1.1}
Let $(\mathbf{v}_{\varepsilon}, c_{\varepsilon})$ be a smooth solution of (\ref{9.6.4})-(\ref{9.6.5}) for some $T_{0}>0$.
For each $\varepsilon\in(0, 1]$, there exists a smooth pair $(\mathbf{{\tilde v}}_A, c_A): \Omega\times\Omega\to\mathbb{R}^2\times\mathbb{R}$.
Moreover, the initial data satisfy
\begin{eqnarray}
\begin{split}\label{11.3.1}
&\|\mathbf{v}_{0,\varepsilon}-\mathbf{{\tilde v}}_A(0)\|_{L^2(\Omega)}^2+\frac{1}{\varepsilon}\|c_{0, \varepsilon}-c_{A}(0)\|_{L^2(\Omega)}^2
+\|c_{0, \varepsilon}-c_{A}(0)\|_{L^4(\Omega)}^4\\
& \quad +\varepsilon^2 \|\nabla (c_{0, \varepsilon}-c_{A}(0))\|_{L^2(\Omega)}^2+\|\{c_{A}(0)(c_{0, \varepsilon}-c_{A}(0))\}\|_{L^2(\Omega)}^2 \\
&\quad + \|\{c_{A}(0)(c_{0, \varepsilon}-c_{A}(0))^3\}\|_{L^1(\Omega)} \le C\varepsilon^4\quad\mbox{ for all }\varepsilon \in(0,1].
\end{split}
\end{eqnarray}
 Then, there are constants $\varepsilon_{0} \in(0,1]$, $R>0$ and $T\in(0, T_{0}]$, such that
\begin{eqnarray}  \label{9.6.1}
\|\mathbf{v}_{\varepsilon}-\mathbf{{\tilde v}}_A\|_{L^{\infty}(0, T; L^{2}(\Omega))}
+ \|\mathbf{v}_{\varepsilon}-\mathbf{{\tilde v}}_A\|_{ L^{2}(0, T ; H^{1}(\Omega))} \leq R\, \varepsilon^{2},
\end{eqnarray}
and
\begin{equation}
\sup_{0 \le t \le T}\|c_{\varepsilon}(t)-c_{A}(t)\|_{L^2(\Omega)}+\|\nabla_\tau (c_{\varepsilon}-c_{A})\|_{L^2(\Omega\times(0,T))}
+\varepsilon  \|\nabla (c_{\varepsilon}-c_{A})\|_{L^2(\Omega\times(0,T))}\le R\,\varepsilon^{\frac{5}{2}},
\end{equation}
\begin{equation}
\sup_{0 \le t \le T} \|\nabla \left(c_{\varepsilon}(t)-c_{A}(t)\right)\|_{L^2(\Omega)}+\varepsilon^{\frac{1}{2}} \|\nabla^2 (c_{\varepsilon}-c_{A})\|_{L^2(\Omega\times(0,T))} \le  R\,\varepsilon \label{11.3.2}
\end{equation}
for all $\varepsilon \in (0,\varepsilon_0]$. Moreover,
\begin{eqnarray}
\begin{split} \label{3.18.1}
&\sup_{0\le t\le T}\| c_{\varepsilon}(t)-c_{A}(t)\|_{L^4(\Omega)}^2 +\sup_{0\le t\le T}\| c_{A}(t)(c_{\varepsilon}-c_{A})(t)\|_{L^2(\Omega)}
  + \|\mu_{\varepsilon}-\mu_A\|_{L^2(\Omega\times(0,t))}  \\
& \quad +\varepsilon^{-\frac{1}{2}} \| f'(c_\varepsilon)-f'(c_A)\|_{L^2(\Omega\times(0,t))} \leq  R \varepsilon^2
\quad\; \mbox{ for all }\varepsilon \in (0,\varepsilon_0],
\end{split}
\end{eqnarray}
where $\mu_{\varepsilon}$ is defined in (\ref{3.18.2}) and $\mu_A=-\varepsilon \Delta c_A+\frac{1}{\varepsilon} f^{\prime}(c_A)$.

Furthermore,
$$
\lim _{\varepsilon \rightarrow 0} c_{A}=\pm 1 \quad \text{ uniformly on compact subsets of } \Omega^{\pm},
$$
and
$$
\mathbf{{\tilde v}}_{A}=\mathbf{v}+O(\varepsilon) \quad \text { in } L^{\infty}(\Omega \times(0, T)) \text { as } \varepsilon \rightarrow 0.
$$
In particular, the above results imply that
$$
c_{\varepsilon} \rightarrow \pm 1 \quad \text { in } L_{\mathrm{loc}}^{2}( \Omega^{\pm}).
$$
\end{thm}
Before proceeding, let us explain the main proof ideas in this paper. First, the construction of the approximate solutions
$\mathbf{{\tilde v}}_{A}$ and $c_{A}$ ensures that $\mathbf{{\tilde v}}_{A}$ and $c_{A}$ converge to $\mathbf{v}$ and $\pm 1$ respectively,
as $\varepsilon$ tends to 0. Then, it suffices to prove (\ref{9.6.1})-(\ref{3.18.1}). We should point out here that
 Abels and Liu recently studied the sharp interface limit for a Stokes/Allen-Cahn coupled system in \cite{Liu Y N}, where
 they established the convergence of $c_\varepsilon$ in the $L^\infty(0, T_0; L^2(\Omega))$ sense and
 the convergence of $\mathbf{v}_{\varepsilon}$ in the $L^2(0, T_0; L^q(\Omega))\ (q\in [1,2))$ sense with well-prepared initial data.
 Here for the Navier-Stokes/Allen-Cahn coupled system, we obtain the $L^\infty(0, T_0; H^1(\Omega))\cap L^2(0, T_0; H^2(\Omega))$ convergence for $c_\varepsilon$ and the $L^\infty(0, T_0; L^2(\Omega))\cap L^2(0, T_0; H^1(\Omega))$ convergence for $\mathbf{v}_{\varepsilon}$.

As mentioned above, this paper contains two main parts: In the first part, we construct the high-order approximate solution
$(\mathbf{{\tilde v}}_{A}, c_{A})$, which solves the original problem (\ref{9.6.4})-(\ref{9.6.5}) with the high order error terms
with respect to $\varepsilon$. Following the arguments in \cite{Liu Y N}, the approximate solution in this paper can be constructed similarly.
Here we require that the approximate solution $\mathbf{{\tilde v}}_{A}$ satisfies the divergence-free condition,
which implies the error function of the velocity also satisfies the same divergence-free condition.
In the second part, the error terms between the exact solution and approximate solution are estimated.
It should be remarked that the most arguments in deriving the $L^2$ estimates of the error functions
are similar to those in \cite{Liu Y N} in some sense. However, to control $\|\nabla u\|^2_{L^2(\Omega\times(0,t))}$,
we shall employ a (slightly) different argument by noticing
$|\nabla u|^2+\frac{1}{\varepsilon^2} f''(c_A) u^2 =(1-\varepsilon^2)(|\nabla u|^2
+\frac{1}{\varepsilon^2} f''(c_A) u^2)+\varepsilon^2 (|\nabla u|^2+\frac{1}{\varepsilon^2} f''(c_A) u^2)$ with $u=c_\varepsilon-c_A$.

In order to derive the estimates of the error functions, the spectrum estimate of the linearized Allen-Cahn operator $\mathcal{L}_\varepsilon=-\Delta+\varepsilon^{-2}f''(c_A)$ is essentially used. It is emphasized that this method was originally used
to study the Allen-Cahn equation in \cite{Chen X}, and also used to the Cahn-Hilliard equation in \cite{Abels}.
To overcome the difficulty caused by the capillary term $\operatorname{div} (\nabla c_{\varepsilon} \otimes \nabla c_{\varepsilon})$,
it is necessary to derive the estimates of the derivatives. However, noticing that there are no corresponding spectral estimates for the second derivatives, we have to handle the singular term $\frac{1}{\varepsilon^{2}}[f'(c_\varepsilon)-f'(c_A)]$.
To this end, Abels and Fei multiplied the equation of the error function for $u$ by $\Delta u$,
and then close the estimates by H\"{o}lder's inequality in a very recent work \cite{Fei},
which will produce a singular factor of $\varepsilon^{-2}$ in estimating this term. However, we come up with a new multiplier of $\mu_{\varepsilon}-\mu_A=-\varepsilon \Delta c_{\varepsilon}+\frac{1}{\varepsilon} f^{\prime}(c_{\varepsilon})+\varepsilon \Delta c_{A}-\frac{1}{\varepsilon} f^{\prime}(c_{A})$ for the error estimate of derivatives. In this way, we do not need to control the term  $\frac{1}{\varepsilon^{2}}[f'(c_\varepsilon)-f'(c_A)] \Delta u$ as in \cite{Fei}. It converts to estimating the term
$\frac{1}{\varepsilon^{2}}[f'(c_\varepsilon)-f'(c_A)] \partial_t u$. Intuitively, $\Delta u$ leads to a singular factor of $\varepsilon^{-2}$,
while $\partial_t u$ only produces $\varepsilon^{-1}$. This will improve the estimate of the error function $\nabla (c_{\varepsilon}-c_A)$
by $\varepsilon^{\frac{1}{2}}$. That is, $\|\nabla \left(c_{\varepsilon}(t)-c_{A}(t)\right)\|_{L^\infty(0,T;L^2(\Omega))}$ is of order $O(\varepsilon)$ stated in (\ref{11.3.2}). This is one of the key observations in this paper.
Moreover, by using this multiplier, some by-products can also be obtained in this paper,
which are listed in the main theorem. We believe that the optimal convergence rate should be $O(\varepsilon^{3/2})$.
However, we do not know how to achieve this optimal rate at this moment, which is left for the future study.

This paper is organized as follows: In Section 2, we give some symbols and elemetary lemmas to be used later. Section 3 is devoted to the construction of the approximate solution which contains the inner and outer expansions, and to the derivation of
the corresponding estimates for the error functions. Finally, in Section 4 we establish the a priori estimates
to complete the proof of the main theorem.

\section{Preliminaries}

For every point $x \in \Gamma_{t}\ (t \in[0, T_{0}])$, there is a local diffeomorphisms $X_{0}: \mathbb{T}^{1} \times [0, T_{0}] \rightarrow \Omega$, such that the normalized tangential vector on $\Gamma_{t}$ at $X_{0}(s, t)$ is described by
$$
\tau(s, t)=\frac{\partial_{s} X_{0}(s, t)}{|\partial_{s} X_{0}(s, t)|}.
$$
Moreover, the outer unit normal vector of interior boundary for $\Omega^{-}(t)$ is denoted as
$$
\mathbf{n}(s, t)=\left(\begin{array}{cc}
0 & -1 \\
1 & 0
\end{array}\right) \tau(s, t), \quad \text { for all }(s, t) \in \mathbb{T}^{1} \times[0, T_{0}].
$$
For convenience, we set
$$
\mathbf{n}_{\Gamma_{t}}(x)\triangleq \mathbf{n}(s, t), \quad \text { for all } x=X_{0}(s, t) \in \Gamma_{t}.
$$
Let $V_{\Gamma_{t}}$ and $H_{\Gamma_{t}}$ be the normal velocity and the (mean) curvature of $\Gamma_{t}$ (with respect to $\mathbf{n}$ ).
By virtue of the definition,
$$
V_{\Gamma_{t}}\left(X_{0}(s, t)\right)=V(s, t)=\partial_{t} X_{0}(s, t) \cdot \mathbf{n}(s, t) \quad \text { for all }(s, t) \in \mathbb{T}^{1} \times [0, T_{0}].
$$
We choose $\delta$ so small that $\operatorname{dist}(\partial \Omega, \Gamma_{t})> 3 \delta$. Then, every
$$
x \in \Gamma_{t}(3 \delta)=\{x \in \Omega: \operatorname{dist}(x, \Gamma_{t})< 3 \delta\},
$$
can be uniquely represented by
$$
x=P_{\Gamma_{t}}(x)+r \mathbf{n}_{\Gamma_{t}}\left(P_{\Gamma_{t}}(x)\right),\quad \text{ where } r=\operatorname{sdist}(\Gamma_{t}, x).
$$

Consider the mapping
$$
X:(-3 \delta, 3 \delta) \times \mathbb{T}^{1} \times[0, T_{0}] \mapsto \Gamma(3 \delta) \text { by } X(r, s, t)\triangleq X_{0}(s, t)+r \mathbf{n}(s, t),
$$
so that $(r, s, t)$ are coordinates in $\Gamma(3 \delta)$, and let
\begin{eqnarray}  \label{9.10.3}
r=\operatorname{sdist}(\Gamma_{t}, x), \quad s=X_{0}^{-1}\left(P_{\Gamma_{t}}(x), t\right)\triangleq S(x, t)
\end{eqnarray}
be its inverse.

Such coordinates are more convenient for the calculations that follow. For instance, noting that
$d_{\Gamma}\left(X_{0}(s, t)+r \mathbf{n}(s, t),\ t\right)=r$, we see that its derivative along $r$ leads to
$$\nabla d_{\Gamma}\left(X_{0}(s, t)+r \mathbf{n}(s, t), t\right)\cdot \mathbf{n}=1, $$
which implies that
\begin{eqnarray} \label{9.10.2}
\nabla d_{\Gamma}(x, t)=\mathbf{n}_{\Gamma_{t}}\left(P_{\Gamma_{t}}(x)\right),\quad \partial_{t} d_{\Gamma}(x, t)=-V_{\Gamma_{t}}\left(P_{\Gamma_{t}}(x)\right),\quad \Delta d_{\Gamma}(p, t)=-H_{\Gamma_{t}}(p),
\end{eqnarray}
for all $(x,t)\in \Gamma(3 \delta)$ and $(p,t)\in \Gamma$ (cf. Section 4.1 in \cite{Chen X}).

Denoting
\begin{eqnarray*}
&\partial_{\tau} u(x, t)\triangleq \tau(S(x, t), t) \nabla u(x, t),
\quad \nabla_{\tau} u(x, t)\triangleq \partial_{\tau} u(x, t) \tau(S(x, t), t), \quad\; (x, t) \in \Gamma(3 \delta),
\end{eqnarray*}
we have
\begin{eqnarray}   \label{9.1.10}
\nabla_{\tau}=(I-\mathbf{n}(S(\cdot), \cdot) \otimes \mathbf{n}(S(\cdot), \cdot)) \nabla.
\end{eqnarray}
Since $\partial_{\mathbf{n}}(I-\mathbf{n} \otimes \mathbf{n})=0$, we find that
\begin{align*}
\left[\partial_{\mathbf{n}}, \nabla_{\tau}\right] g &\triangleq \partial_{\mathbf{n}}((I-\mathbf{n} \otimes \mathbf{n})
\nabla g)-(I-\mathbf{n} \otimes \mathbf{n}) \nabla\left(\partial_{\mathbf{n}} g\right) \\
&=(I-\mathbf{n} \otimes \mathbf{n}) \partial_{\mathbf{n}} \nabla g-(I-\mathbf{n} \otimes \mathbf{n}) \nabla(\mathbf{n} \cdot \nabla g) \\
&=-\sum_{j=1}^{2}\left((I-\mathbf{n} \otimes \mathbf{n}) \nabla \mathbf{n}_{j}\right) \partial_{x_{j}} g
=-\tau\left(\partial_{\tau} \mathbf{n} \cdot \nabla g\right),
\end{align*}
which shows that the commutator $\left[\partial_{\mathbf{n}},\nabla_{\tau}\right]$ is in fact a tangential differential operator
(cf. Section 2.2 in \cite{Liu Y N}).

In this paper, we shall identify a function $\omega(x, t)$ with $\tilde{\omega}(r, s, t)$, such that
$$
\omega(x, t)=\tilde{\omega}\left(d_{\Gamma}(x, t), S(x, t), t\right),\;\; \text { namely }  \omega\left(X_{0}(s, t)+r \mathbf{n}(s, t), t\right)=\tilde{\omega}(r, s, t). $$
By using the chain rule together with (\ref{9.10.2}), we have the following formula
\begin{eqnarray}
\begin{split}  \label{9.11.1}
&\partial_{t} \omega(x, t)=-V_{\Gamma_{t}}\left(P_{\Gamma_{t}}(x)\right) \partial_{r} \tilde{\omega}(r, s, t)
+\partial_{t}^{\Gamma} \tilde{\omega}(r, s, t), \\
&\nabla \omega(x, t)=\mathbf{n}_{\Gamma_{t}}\left(P_{\Gamma_{t}}(x)\right) \partial_{r} \tilde{\omega}(r, s, t)
+\nabla^{\Gamma} \tilde{\omega}(r, s, t), \\
&\Delta \omega(x, t)=\partial_{r}^{2} \tilde{\omega}(r, s, t)+\Delta d_{\Gamma_{t}}(x) \partial_{r} \tilde{\omega}(r, s, t)
+\Delta^{\Gamma} \tilde{\omega}(r, s, t),
\end{split}
\end{eqnarray}
where $r, s$ are defined by (\ref{9.10.3}), and
\begin{eqnarray}
\begin{split}  \label{9.12.1}
\partial_{t}^{\Gamma}\tilde{\omega}(r,s,t) &=\partial_{t} \tilde{\omega}(r,s,t)+\partial_{t} S(x,t)\partial_{s}\tilde{\omega}(r,s,t), \\
\nabla^{\Gamma} \tilde{\omega}(r, s, t) &=(\nabla S)(x, t) \partial_{s} \tilde{\omega}(r, s, t), \\
\Delta^{\Gamma} \tilde{\omega}(r, s, t) &=|(\nabla S)(x, t)|^{2} \partial_{s}^{2} \tilde{\omega}(r, s, t)
+(\Delta S)(x, t) \partial_{s} \tilde{\omega}(r, s, t).
\end{split}
\end{eqnarray}
When $\Gamma$ is smooth enough, then $|\nabla S|\le C$, cf. Section 4.1 in \cite{Chen X}.

As in the previous construction, the leading term $c^{in}_0(x,t)\triangleq \theta_0(\rho)$ of $c_A$ is a function of the stretched
variable $\rho$, which is defined as follows.
\begin{eqnarray}  \label{9.11.2}
\rho(x, t) \triangleq \frac{d_{\Gamma}(x, t)}{\varepsilon}-h_{\varepsilon}(S(x, t), t), \quad h_{\varepsilon}(s, t) \triangleq h_{1}(s, t)+\varepsilon h_{2}(s, t).
\end{eqnarray}
The reason why the factor $h_{\varepsilon}$ is introduced in the definition is to circumvent the obstacles and difficulties caused
by the error of $c_\varepsilon$, which will be discussed later.

Set
$$\|\psi\|_{L^{p, \infty}\left(\Gamma_{t}(2 \delta)\right)}\triangleq \left(\int_{\mathbb{T}^{1}} \operatorname{ess} \sup _{|r| \leq 2 \delta}\left|\psi \left(X_{0}(s, t)+r \mathbf{n}(s, t)\right)\right|^{p} \mathrm{~d} s\right)^{\frac{1}{p}}.
$$

Denote the function space
\begin{eqnarray}   \label{9.6.3}
X_{T}\triangleq L^{2}\left(0, T ; H^{5 / 2}(\mathbb{T}^{1})\right) \cap H^{1}\left(0, T ; H^{1 / 2}(\mathbb{T}^{1})\right)
\end{eqnarray}
equipped with the following norm
$$
\|u\|_{X_{T}}=\|u\|_{L^{2}(0, T ; H^{5 / 2}(\mathbb{T}^{1}))}+\|u\|_{H^{1}(0, T ; H^{1 / 2}(\mathbb{T}^{1}))}+\|u|_{t=0}\|_{H^{3 / 2}(\mathbb{T}^{1})}.
$$
Recalling
\begin{eqnarray}   \label{11.16.2}
X_{T} \hookrightarrow B U C\left([0, T] ; H^{3 / 2}(\mathbb{T}^{1})\right) \cap L^{4}\left(0, T ; H^{2}(\mathbb{T}^{1})\right),
\end{eqnarray}
one sees that the operator norm of the embedding is uniformly bounded in $T$.

Let $X_{T}$ be the function space defined in (\ref{9.6.3}). For $h_{1}$, $h_{2}=h_{2, \varepsilon}$ presented above,
we require the following a priori assumptions:
\begin{eqnarray}  \label{9.6.2}
h_{1} \in C^{\infty}(\mathbb{T}^{1} \times[0, T_{0}]),\;\; \sup _{0<\varepsilon \leq \varepsilon_{0}}\|h_{2, \varepsilon}\|_{X_{T}} \leq M
\;\mbox{ for some }\varepsilon_{0} \in(0,1), M \geq 1,\ T \in (0, T_{0}],
\end{eqnarray}
where we keep in mind that only $h_{2}=h_{2,\varepsilon}$ depends on $\varepsilon$.
The a priori assumptions in (\ref{9.6.2}) will be verified later.

To describe the properties of the leading term $\theta_0$ of $c_A$, which depends on the stretch variable $\rho$,
it is convenient to introduce the following function spaces.
\begin{defn}
 For any $k \in \mathbb{R}$ and $\alpha>0$, $\mathcal{R}_{k, \alpha}$ is the space of functions
 $\hat{r}_{\varepsilon}: \mathbb{R} \times \Gamma(2 \delta) \rightarrow \mathbb{R}$, $\varepsilon \in(0,1)$, such that
$$
|\partial_{\mathbf{n}_{{\Gamma}_{t}}}^{j} \hat{r}_{\varepsilon}(\rho, x, t)| \leq C e^{-\alpha|\rho|}\varepsilon^{k}, \quad \text { for all } \rho \in \mathbb{R},(x, t) \in \Gamma(2 \delta), j=0,1, \varepsilon \in(0,1),
$$
where $C>0$ is some constant independent of $(\rho,x, t)$ and $\varepsilon$, and the equipped norm $\|\cdot\|_{\mathcal{R}_{k,\alpha}}$
can be defined as
$$
\|(\hat{r}_{\varepsilon})_{\varepsilon \in(0,1)}\|_{\mathcal{R}_{k, \alpha}}=\sup _{\varepsilon \in(0,1),(x, t)
\in \Gamma(2 \delta), \rho \in \mathbb{R}, j=0,1}|\partial_{n_{ \Gamma_{t}}}^{j} \hat{r}_{\varepsilon}(\rho, x, t)| e^{\alpha|\rho|}
\varepsilon^{-k}.
$$
Besides, we regard $r_{\varepsilon}(x, t)$ as $\hat{r}_{\varepsilon}\left(\frac{d_{\Gamma}(x, t)}{\varepsilon}-h_{\varepsilon}(S(x, t), t), x, t\right) $ for all $(x, t) \in \Gamma(2 \delta)$. Finally, $(\hat{r}_{\varepsilon})_{\varepsilon \in(0,1)}\in \mathcal{R}_{k, \alpha}^0$ means that $\hat{r}_{\varepsilon}$ have value-zero on $x \in \Gamma_t$ in the usual sense.
\end{defn}

Based on the above definition, we have the following lemma.
\begin{lem}
Under the a priori assumptions (\ref{9.6.2}), set
$M \triangleq \sup _{0<\varepsilon \le \varepsilon_{0},(s, t) \in \mathbb{T}^{1} \times [0, T]}|h_{\varepsilon}(s, t)|<\infty.$
Let $(\hat{r}_{\varepsilon})_{\varepsilon \in(0,1)} \in \mathcal{R}_{k, \alpha}$. Then,
\begin{eqnarray} \label{10.18.2}
\big\|\sup _{(x, t) \in \Gamma(2 \delta)}|\hat{r}_{\varepsilon}(\frac{.}{\varepsilon}, x, t)|\big\|_{L^{2}(\mathbb{R})} \leq C \varepsilon^{k+\frac{1}{2}},
\end{eqnarray}
where the positive constant C is independent of $M$, $T$ and $\varepsilon$.
\end{lem}
\begin{proof}
In view of the exponential decay properties of $\hat{r}_{\varepsilon}$, we have
$$
\big\|\sup _{(x, t) \in \Gamma(2 \delta)}|\hat{r}_{\varepsilon}(\cdot, x, t)|\big\|_{L^{2}(\mathbb{R})}^{2} \leq C \varepsilon^{2k} \int_{-\infty}^{\infty} e^{-2\alpha |r| / \varepsilon} \mathrm{d} r=C \varepsilon^{2k+1} \int_{-\infty}^{\infty} e^{-2\alpha |z|} \mathrm{d} z=C \varepsilon^{2k +1}.
$$
\end{proof}

\begin{rem} \label{rem11.16.1}
(1) The $L^2$-norm in (\ref{10.18.2}) can be replaced by $L^p$-norm for any $1 \leq p \leq \infty$, and correspondingly, the right hand side will become $\varepsilon^{k+\frac{1}{p}}$.\\
(2) Moreover, if $\left(\hat{r}_{\varepsilon}\right)_{\varepsilon \in(0,1)} \in \mathcal{R}_{k, \alpha}^{0}$, then there is a constant $C>0$, depending on $M$, such that
$$
\big\|\sup _{(x, t) \in \Gamma(2 \delta)}|\hat{r}_{\varepsilon}(\frac{.}{\varepsilon}, x, t)|\big\|_{L^{p}(\mathbb{R})} \leq C(M) \varepsilon^{k+\frac{1}{p}+1} .
$$
\end{rem}

\begin{prop} \label{Proposition:prop1.4}
Suppose $\hat{r}_{\varepsilon} \in \mathcal{R}_{0, \alpha}$ for some $\alpha>0$. Let $j=0$ if $\hat{r}_{\varepsilon} \in \mathcal{R}_{0, \alpha}$
and $j=1$ if $\hat{r}_{\varepsilon} \in \mathcal{R}^0_{0, \alpha}$. Then, under the assumptions (\ref{9.6.2}), there is a $C=C(M)>0$, such that
\begin{align*}
\|a(P_{\Gamma_{t}}(\cdot)) r_{\varepsilon } \varphi\|_{L^{1}(\Gamma_{t}(2 \delta))} \le&\, C\varepsilon^{1+j} \|\varphi\|_{H^{1}(\Omega)}\|a\|_{L^{2}(\Gamma_{t})}, \\
\|a (P_{\Gamma_{t}}(\cdot)) r_{\varepsilon}\|_{L^{2}(\Gamma_{t}(2 \delta))} \le& \,C \varepsilon^{\frac{1}{2}+j}\|a\|_{L^{2}(\Gamma_{t})}
\end{align*}
uniformly for all $\varphi \in H^{1}(\Omega)$ and $a \in L^{2}(\Gamma_{t})$.
\end{prop}
\begin{proof}
By the coordinate transformation $(x, t) \mapsto(r, s, t)$, we obtain
\begin{align*}
&\|a(P_{\Gamma_{t}}(\cdot)) r_{\varepsilon} \varphi\|_{L^{1}(\Gamma_{t}(2 \delta))} \\
=&\;\int_{-2 \delta}^{2 \delta} \int_{\mathbb{T}^{1}}|a(X_{0}(s, t))| \left|\hat{r}_{\varepsilon}\left(\frac{r}{\varepsilon}
-h_{\varepsilon}(s, t), X(r, s, t), t\right)\right|\,|\varphi(X(r, s, t))| J(r, s) \mathrm{d} s \mathrm{~d} r \\
\le&\; C \|a(X_{0}(s, t))\|_{L^{2, \infty}\left(\Gamma_{t}(2 \delta)\right)}\left\|\sup _{(x, t) \in \Gamma(2 \delta)}|\hat{r}_{\varepsilon}(\frac{\cdot}{\varepsilon}, x, t)|\right\|_{L^{1}(\mathbb{R})} \|\varphi(X(r, s, t))
\|_{L^{2, \infty}\left(\Gamma_{t}(2 \delta)\right)} \\
\le & \; C \varepsilon^{1+j} \|\varphi\|_{H^{1}(\Omega)}\|a\|_{L^{2}(\Gamma_{t})}
\end{align*}
for all $a \in L^{2}\left(\Gamma_{t}\right)$ and $\varphi \in H^{1}(\Omega)$, which implies the first inequality.

The second inequality can be gotten as follows, using a straightforward calculation.
\begin{align*}
&\|a(P_{\Gamma_{t}}(\cdot)) r_{\varepsilon}\|_{L^{2}(\Gamma_{t}(2 \delta))} =\left(\int_{-2 \delta}^{2 \delta} \int_{T^{1}}|a(X_{0}(s, t))|^{2} |r_{\varepsilon}(X(r, s, t), t))|^{2} J(r, s) \mathrm{d} s \mathrm{d} r\right)^{\frac{1}{2}} \\
\le& \;C\Big(\|a(X_{0}(s, t))\|^{2}_{L^{2, \infty}\left(\Gamma_{t}(2 \delta)\right)}\Big\|\sup _{(x, t) \in \Gamma(2 \delta)}|\hat{r}_{\varepsilon}(\frac{\cdot}{\varepsilon}, x, t)|\Big\|^{2}_{L^{2}(\mathbb{R})} \Big)^{\frac{1}{2}} \le C \varepsilon^{\frac{1}{2}+j} \|a\|_{L^{2}(\Gamma_{t})}
\end{align*}
for all $a \in L^{2}(\Gamma_{t})$.
\end{proof}

The following Gagliardo-Nirenberg inequality will be used frequently.
\begin{lem}
\label{GN}
(Gagliardo-Nirenberg inequality \cite{Nirenberg}) Let $u$ be a suitable function defined in ${\mathbb R}^n$.
For any $1 \leq q, r \leq \infty$ and a natural number $m$, $\alpha$ and $j$ satisfy
$$
\frac{1}{p}=\frac{j}{n}+\left(\frac{1}{r}-\frac{m}{n}\right) \alpha+\frac{1-\alpha}{q},
\qquad \frac{j}{m} \leq \alpha \leq 1.
$$
Then, there are positive constants $C_1$ and $C_2$ depending only on $m, n, j, q , r$ and $\alpha$, such that
$$
\|\mathrm{D}^{j} u\|_{L^p} \leq C_1\|\mathrm{D}^{m} u\|_{L^r}^{\alpha} \|u\|_{L^{q}}^{1-\alpha} +C_2 \|u\|_{L^{q}}.
$$
\end{lem}
\noindent A special but important case of the above lemma reads as
\begin{rem}
$\|u\|_{L^\infty} \le C\|u\|_{L^2}^{\frac{1}{2}}\|u\|_{H^1}^{\frac{1}{2}}$ follows from taking $m=n=1$, $j=0$, $p=\infty$ and $r=q=2$.
\end{rem}

For any $t \in [0,T]$, any small $\varepsilon>0$, and the approximate solution $c_A$, the spectrum of the self-adjoint operator $\mathcal{L_\varepsilon}=-\Delta+\varepsilon^{-2}f''(c_A)$ has a lower bound. More precisely, we have the following estimate.
\begin{prop} \label{theorem:thm2.10}
Let $c_{A}$ be the approximate solution, and the a priori assumptions $(\ref{9.6.2})$ be satisfied for some $M>0$.
Then, there are constants $C, \varepsilon_{0}>0$, independent of $M$ and $c_{A}$, such that for every $t \in [0, T_{0}]$
and $\varepsilon \in(0, \varepsilon_{0}]$,
$$
\int_{\Omega}\left(|\nabla \varphi(x)|^{2}+\varepsilon^{-2} f^{\prime \prime}\left(c_{A}(x, t)\right) \varphi^{2}(x)\right) \mathrm{d} x
\geq-C \int_{\Omega} \varphi^{2} \mathrm{~d} x+\int_{\Omega }|\nabla_{\tau} \varphi|^{2} \mathrm{~d}x,\quad\forall\, \varphi \in H^{1}(\Omega).
$$
\end{prop}
\begin{proof}
The proof can be found in \cite[Theorem 2.13]{Liu Y N}.
\end{proof}

Finally, it is convenient to introduce the following property
which follows directly from the construction of the approximate solutions.
\begin{lem}   \label{corollary:corB.3}
Let the assumptions (\ref{9.6.2}) be satisfied, and $\hat{c}_{2}$ and $\hat{c}_{3}$ be defined in (\ref{9.27.2}). Then, we have
\begin{align*}
&\varepsilon^{2} \|(c_{2}^{i n},\nabla_{\tau} c_{2}^{i n})\|_{L^\infty\left(0, T; L^{2}\left(\Gamma_{t}(2 \delta)\right)\right)} \leq C(M) \varepsilon^{\frac{5}{2}}, \quad \varepsilon^{2}\| \partial_{\mathbf{n}} c_{2}^{i n} \|_{L^{\infty}\left(0, T; L^{2}\left(\Gamma_{t}(2 \delta)\right)\right)} \leq C(M) \varepsilon^{\frac{3}{2}}, \\
&\varepsilon^{3}\|\nabla c_{3}^{i n}\|_{L^{\infty}\left(0, T; L^{2}\left(\Gamma_{t}(2 \delta)\right)\right)} \leq C(M, \theta) \varepsilon^{\frac{5}{2}-\theta}, \quad \varepsilon^{3}\|\nabla c_{3}^{i n}\|_{L^{2}\left(0, T; L^{2}\left(\Gamma_{t}(2 \delta)\right)\right)} \leq C(M) \varepsilon^{\frac{5}{2}},
\end{align*}
for any $\theta \in(0,1)$.
\end{lem}
\begin{proof}
We refer to \cite[Lemma 4.3]{Liu Y N} for a proof of this lemma.
\end{proof}


\section{Estimates of Solutions to the Error Equations}
In this section, based on the matched asymptotical expansion method, we first construct the approximate solution
used in this paper. Then, we derive the estimates of the error functions by the a priori energy estimate method.
For $p_A$ and $c_A$, we shall adopt the construction in \cite{Liu Y N}, while for $\mathbf{\tilde{v}}_A$ we shall take
a small adjustment to make it satisfy the divergence-free condition.

\subsection{Construction of the approximate solution}
The approximate solution contains two main parts: the inner layer part and the outer part, which are constructed
by the matched asymptotic expansion method. First, we define the inner approximate solution as follows.
\begin{eqnarray}
\begin{split}   \label{9.27.2}
\mathbf{v}_A^{in}(\rho,x,t)&=\mathbf{v}_0(\rho,x,t)+\varepsilon \mathbf{v}_1(\rho,x,t)+\varepsilon^2 \mathbf{v}_2(\rho,x,t),\\
p_A^{in}(\rho,x,t)&=\varepsilon^{-1} p_{-1}(\rho,x,t)+p_0(\rho,x,t)+\varepsilon p_1(\rho,x,t),\\
c^{i n}(x, t)&=\hat{c}^{i n}(\rho, s, t) =\theta_{0}(\rho)+\varepsilon^{2} \hat{c}_{2}(\rho, S(x, t), t)+\varepsilon^{3} \hat{c}_{3}(\rho, S(x, t), t) \\
&\triangleq c_{0}^{i n}(x, t)+\varepsilon^{2} c_{2}^{i n}(x, t)+\varepsilon^{3} c_{3}^{i n}(x, t),
\end{split}
\end{eqnarray}
where recall that $s=S(x, t)$, and
$$
\rho=\frac{d_{\Gamma}(x, t)}{\varepsilon}-h_{1}(S(x, t), t)-\varepsilon h_{2, \varepsilon}(S(x, t), t).
$$
To write the formula for $\mathbf{v}_{i}$ and $p_{i}$, we require
$\eta(\rho)\triangleq -1+\frac{2}{\sigma} \int_{-\infty}^{\rho}\theta_{0}^{\prime}(s)^{2}\mathrm{~d} s$ for all $\rho\in\mathbb{R}$,
which means
$$
|\eta(\rho) \mp 1| = O(e^{-\alpha|\rho|}), \ \quad \text { when }\ \rho \gtrless 0,
$$
for some $\alpha>0$. Let $\mathbf{v}_{i}$ and $p_{j}$ take the following form:
\begin{align*}
\mathbf{v}_{i} (\rho, x, t)&=\tilde{\mathbf{v}}_{i}(\rho, x, t)+\eta(\rho) d_{\Gamma}(x, t) \hat{\mathbf{v}}_{i}(x, t),\quad i=0,1,2,\\
p_{j}(\rho, x, t)&=\tilde{p}_{j}(\rho, x, t)+\eta(\rho) d_{\Gamma}(x, t) \hat{p}_{j}(x, t),  \quad j=-1,0,1,
\end{align*}
where $\tilde{\mathbf{v}}_{i}$, $\hat{\mathbf{v}}_{i}$, $\tilde{p}_{j}$ and $\tilde{p}_{j}$ are defined as those in \cite[Section 3.1]{Liu Y N}.

To get the approximate solution, we require the outer expansion satisfies
\begin{align*}
\mathbf{v}_A^{\pm}(x,t)=\mathbf{v}_0^{\pm}(x,t)+\varepsilon \mathbf{v}_1^{\pm}(x,t)+\varepsilon^2 \mathbf{v}_2^{\pm}(x,t),\quad p_A^{\pm}(x,t)=p_0^{\pm}(x,t)+\varepsilon p_1^{\pm}(x,t),\quad c_{\pm}^{out}=\pm 1.
\end{align*}
Here $ p_{-1}^{\pm} \equiv 0$, $\mathbf{v}_{0}^{\pm}$ and $p_{0}^{\pm}$ are defined by
$\mathbf{v}_{0}^{\pm} \triangleq \mathbf{v}|_{\Omega^{\pm}(t)}$ and $p_{0}^{\pm}\triangleq p|_{\Omega^{\pm}(t)}$ respectively,
where $(\mathbf{v}, p)$ is the smooth solution of (\ref{11.8.1})-(\ref{11.8.2}).
Moreover, $\mathbf{v}_{1}^{\pm}$, $\mathbf{v}_{2}^{\pm}$ and $p_{1}^{\pm}$ are defined in the same way as
in \cite[Section 3.1]{Liu Y N}. In addition, we select a smooth cut-off function $\zeta$ satisfying
\begin{eqnarray*}
\begin{cases}
\zeta(r) \equiv 1 \quad \text{ if } |r| \le \delta,\quad \zeta(r) \equiv 0 \quad \text{ if } |r| \ge 2\delta,\\
0\le-r\zeta'(r)\le4 \quad \text{ if } \delta \le |r| \le2\delta.
\end{cases}
\end{eqnarray*}
In summary, we ``glue" the internal and external expansions together to construct the approximate solution $(\mathbf{v}_A,p_A,c_{A})$
in $\Omega \times [0, T_{0}]$ as
\begin{align}
\mathbf{v}_A(x,t)&=\zeta \circ d_\Gamma \mathbf{v}_A^{in}(\rho,x,t)+(1-\zeta \circ d_\Gamma)(\mathbf{v}_A^+(x,t)\chi_+ +\mathbf{v}_A^-(x,t)\chi_-), \label{9.15.1}\\
p_A(x,t)&=\zeta \circ d_\Gamma p_A^{in}(\rho,x,t)+(1-\zeta \circ d_\Gamma)(p_A^+(x,t)\chi_+ +p_A^-(x,t)\chi_-),\notag\\
c_{A}(x, t) &=\zeta \circ d_{\Gamma} c^{i n}(x, t)+(1-\zeta \circ d_{\Gamma})(c_{+}^{o u t} \chi_{+}+c_{-}^{o u t} \chi_{-}) \notag\\
&=c_{+}^{out } \chi_{+}+c_{-}^{out } \chi_{-}+\left(c^{i n}(x, t)-c_{+}^{out} \chi_{+}-c_{-}^{out} \chi_{-}\right) \zeta \circ d_{\Gamma},\label{9.27.3}
\end{align}
where $\chi_{\pm}\triangleq \chi_{\Omega^{\pm}(t)}$.

Lastly, let us assume that $\mathbf{\tilde{v}}_A$ takes the form of $\mathbf{\tilde{v}}_A=\mathbf{v}_A+\varepsilon \,\hat{\mathbf{f}}$.
As in (\ref{9.15.1}), we have
\begin{align*}
\operatorname{div}\mathbf{v}_A
= &\;\operatorname{div}(\zeta \circ d_\Gamma)(\mathbf{v}_A^{in}-\mathbf{v}_A^+\chi_+ -\mathbf{v}_A^-\chi_-)
+(\zeta \circ d_\Gamma)\operatorname{div} \,\mathbf{v}_A^{in} \\
&  +(1-\zeta \circ d_\Gamma)(\operatorname{div} \,\mathbf{v}_A^+\chi_+ +\operatorname{div} \,\mathbf{v}_A^-\chi_-)
\triangleq I_1+I_2+I_3.
\end{align*}
For $I_1$, one sees $\operatorname{div}(\zeta \circ d_\Gamma)=0$ within $\Gamma(\delta) $; while outside $\Gamma(\delta)$, $\mathbf{v}_A^{in}-\mathbf{v}_A^+\chi_+ -\mathbf{v}_A^-\chi_-$ decays exponentially with respect to the stretched variable $\rho$.
Accordingly, $\mathbf{v}_A^{in}-\mathbf{v}_A^+\chi_+ -\mathbf{v}_A^-\chi_- \sim O(\varepsilon^2)$.
Thus, we infer from the matched asymptotic expansion of divergence equation that
\begin{align}
\operatorname{div} \,\mathbf{v}_A^{in}= \varepsilon &\big(-(\rho+h_{1}) \eta^{\prime}(\rho) \nabla^{\Gamma} h_{1} \cdot
\hat{\mathbf{v}}_{0, \tau}+h_{2} \eta^{\prime}(\rho) \hat{\mathbf{v}}_{0, \mathbf{n}}+(\rho+h_{1}) \eta^{\prime}(\rho)
\hat{\mathbf{v}}_{1, \mathbf{n}}\notag\\
&-d_{\Gamma} \eta^{\prime}(\rho) \nabla^{\Gamma} h_{1} \cdot \hat{\mathbf{v}}_{1, \tau}-d_{\Gamma} \eta^{\prime}(\rho) \nabla^{\Gamma} h_{2}
\cdot \hat{\mathbf{v}}_{0, \tau}\big)\label{9.25.1}\\
+\varepsilon^{2}&\big((\rho+h_{1}) h_{2} \eta^{\prime}(\rho) \frac{\hat{\mathbf{v}}_{0, \mathbf{n}}}{d_{\Gamma}}-h_{2} \eta^{\prime}(\rho)
\nabla^{\Gamma} h_{1} \cdot \hat{\mathbf{v}}_{0, \tau}+h_{2} \eta^{\prime}(\rho) \hat{\mathbf{v}}_{1, \mathbf{n}}-d_{\Gamma} \eta^{\prime}(\rho)
\nabla^{\Gamma} h_{2} \cdot \hat{\mathbf{v}}_{1, \tau}\notag \\
&-\nabla^{\Gamma} h_{\varepsilon} \cdot \partial_{\rho} \tilde{\mathbf{v}}_{2, \tau}+(\rho+h_{\varepsilon}) \eta^{\prime}(\rho)
\hat{\mathbf{v}}_{2, \mathbf{n}}(x, t)-d_{\Gamma} \eta^{\prime}(\rho) \nabla^{\Gamma}
h_{\varepsilon} \cdot \hat{\mathbf{v}}_{2, \tau}(x, t)\notag \\
&+\operatorname{div}\left(\tilde{\mathbf{v}}_{2}(\rho, x, t)+\hat{\mathbf{v}}_{2}({x}, t) d_{\Gamma} \eta(\rho)\right)\big), \notag
\end{align}
where the detailed calculations are omitted for the sake of simplicity, and can be found in \cite[Appendix]{Liu Y N}.

Since (\ref{9.25.1}) vanishes on $\Gamma$ and $d_\Gamma=\varepsilon(\rho+h_\varepsilon)$,
we replace ${d_\Gamma}$ by $\varepsilon (\rho+h_\varepsilon)$ in (\ref{9.25.1}). Moreover, by virtue of $\eta^{\prime}(\rho) $,  $I_2$
can be viewed as power of $\varepsilon^2$. To proceed further, we obtain
$I_3=\varepsilon^2(1-\zeta \circ d_\Gamma)\operatorname{div} \,\mathbf{v}_2^{\pm} \chi_{\pm} $.
Hence, an appropriate $\mathbf{f}$ can be selected, such that
\begin{eqnarray}   \label{9.27.1}
\mathbf{\tilde{v}}_A=\mathbf{v}_A+\varepsilon^2 \,\mathbf{f},
\end{eqnarray}
where $\mathbf{\tilde{v}}_A$ is required to satisfy the divergence free condition, which is useful in computing the error function
of the pressure $p$.

\subsection{Estimates of the error equation for the velocity}
In this subsection, we consider the estimates of the error function of the velocity.
Let $(\mathbf{v}_\varepsilon, p_\varepsilon)$ be a solution to the equation
\begin{eqnarray} \label{2.1}
\partial_t\mathbf{v}_\varepsilon+\mathbf{v}_\varepsilon \nabla \mathbf{v}_\varepsilon-\Delta \mathbf{v}_\varepsilon+\nabla p_\varepsilon
+\varepsilon \operatorname{div}(\nabla c_\varepsilon \otimes \nabla c_\varepsilon)=0.
\end{eqnarray}
Based on the construction in the previous subsection, we can carry out calculations similar to those in the proof of
\cite[Theorem 3.5]{Liu Y N} to find that the approximate solution $(\mathbf{{\tilde v}}_A, p_A)$ satisfies
\begin{eqnarray}
\begin{split}  \label{2.2}
&\partial_t\mathbf{{\tilde v}}_A+\mathbf{{\tilde v}}_A \nabla \mathbf{{\tilde v}}_A-\Delta \mathbf{{\tilde v}}_A+\nabla p_A+\varepsilon \,\operatorname{div}(\nabla c_A \otimes \nabla c_A)\\
& \quad =(\zeta \circ d_\Gamma)\Theta_1+\varepsilon^3 \operatorname{div}[(\zeta \circ d_\Gamma)\Theta_2]+\Theta_3+\varepsilon \operatorname{div}\Theta_4,
\end{split}
\end{eqnarray}
where the lower-order term $\Theta_1$ decays exponentially with respect to the stretched variable,
$\Theta_3$ and $\Theta_4$ are the higher-order error terms. Moreover,
\begin{align} \label{4.1.1}
\Theta_2=\nabla c_0^{in}\otimes \nabla g+\nabla g \otimes \nabla c_0^{in} \quad  \text{ and } \quad \|(\Theta_3, \Theta_4)\|_{L^\infty(0,T;L^2(\Omega))} \le C\varepsilon^2,
\end{align}
with $g=c_2^{in}+\varepsilon c_3^{in}$. From Lemma \ref{corollary:corB.3} it follows that
\begin{align} \label{11.16.1}
\|\partial_\mathbf{n}^i g\|_{L^\infty(0,T;L^2(\Omega))} \le C\varepsilon^{\frac{1}{2}-i},\quad \|\nabla_\tau g\|_{L^\infty(0,T;L^2(\Omega))}
\le C\varepsilon^{\frac{1}{2}} .
\end{align}

Assume $(\mathbf{v}_A^{in},p_A^{in},c^{in})$ takes the form in (\ref{9.27.2}). Then, inserting $(\mathbf{v}_A^{in},p_A^{in},c^{in})$
into the equation of (\ref{9.6.4}), we find that there is no essential difference in the expansions
between the Stokes equations considered in \cite{Liu Y N} and the Navier-Stokes equations here, at least for the expansions
up to the order of $\varepsilon^{1}$. So, we are able to follow a process similar to that used in \cite[Appendix]{Liu Y N}
and utilize Lemma 3.4 in \cite{Liu Y N} to deduce that
\begin{eqnarray}
\begin{split}  \label{12.21.1}
\Theta_1 \triangleq &\partial_t \mathbf{v}_{A}^{i n}+\mathbf{v}_{A}^{i n}\cdot\nabla \mathbf{v}_{A}^{i n}-\Delta \mathbf{v}_{A}^{i n}+\nabla p_{A}^{i n}+\varepsilon \operatorname{div}(\nabla c_{0}^{i n} \otimes \nabla c_{0}^{i n})\\
=&\,r_{\varepsilon}\left(\frac{d_{\Gamma}}{\varepsilon}-h_{\varepsilon}, x, t\right)+\varepsilon \,\tilde{r}_{\varepsilon}\left(\frac{d_{\Gamma}}{\varepsilon}-h_{\varepsilon}, x, t\right) +\sum_{i\leq 2 ; 0 \leq i', j, j^{\prime} \leq 1} \varepsilon^{2} R_{\varepsilon}^{i j i^{\prime} j^{\prime}}(x, t)(\partial_{s}^{i} h_{2})^{i^{\prime}}(\partial_{s}^{j} h_{2})^{j^{\prime}} \\
&+\sum_{0 \leq i,j,k, i^{\prime}, j^{\prime}, k^{\prime} \leq 1} \varepsilon^{2} \tilde{R}_{\varepsilon}^{i j k i^{\prime} j^{\prime} k^{\prime}}(x, t)(\partial_{s}^{i} h_{2})^{i^{\prime}}(\partial_{s}^{j} h_{2})^{j^{\prime}}(\partial_{s}^{k} h_{2})^{k^{\prime}}\quad\;\;
\mbox{in }\Gamma_{t}(3\delta),
\end{split}
\end{eqnarray}
 where $h_{\varepsilon}$ is defined by (\ref{9.6.2}) and the other functions satisfy the following properties:
\begin{align}  \label{9.10.1}
\;(r_{\varepsilon})_{0<\varepsilon<1} \in \mathcal{R}_{0, \alpha}^{0}, \;\; (\tilde{r}_{\varepsilon})_{0<\varepsilon<1}
\in \mathcal{R}_{0, \alpha}\quad \text{and} \quad \|(R_{\varepsilon}^{i j i^{\prime} j^{\prime}},
\tilde{R}_{\varepsilon}^{i j k i^{\prime} j^{\prime} k^{\prime}})\|_{L^\infty((0,T)\times\Gamma(3 \delta))} \le C
\end{align}
for some $\alpha,C >0  $.

To get the estimate stated in (\ref{9.6.1}), we shall proceed through this subsection to derive a bound for the term
 ${\sup}\, \|\mathbf{w}\|^2_{L^2(\Omega)}$. Set
\begin{eqnarray} \label{9.1.3}
\mathbf{w}=\mathbf{v}_\varepsilon-\mathbf{\tilde v}_A, \quad u =c_\varepsilon-c_A.
\end{eqnarray}

\begin{prop}
\label{P3.1}
Let $(\mathbf{w},u)$ be defined by (\ref{9.1.3}), and the assumptions $(\ref{9.6.2})$ be satisfied.
Then, there are a constant $C(M)>0$ and a suitably small constant $\eta>0$, such that for any $t\in(0,T)$,
the following inequality holds.
\begin{eqnarray}
\begin{split}   \label{9.1.4}
&\frac{1}{2} \|\mathbf{w}(t)\|_{L^2(\Omega)}^2+\frac{3}{4} \int_0^t \|\nabla \mathbf{w}\|_{L^2(\Omega)}^2 \mathrm{d}\varsigma\\
\leq& \;\frac{1}{2} \|\mathbf{w}(0)\|_{L^2(\Omega)}^2+C \|\mathbf{w}\|_{L^2(\Omega\times(0,t))}^2+\eta \varepsilon \|\nabla u\|^2_{L^2(\Omega\times(0,t))}\\
&+\eta \frac{ \|\nabla_\tau u\|^2_{L^2(\Omega\times(0,t))}}{\varepsilon }+C\varepsilon \|\nabla u\|_{L^{\infty}(0, t ; L^{2}(\Omega))}^4+C\varepsilon^3 \|\Delta u\|_{L^{2}(\Omega \times(0, t))}^4+C\varepsilon^4 .
\end{split}
\end{eqnarray}
\end{prop}
\begin{proof}
Recalling that
\begin{eqnarray}
\begin{split} \label{2.25.1}
\nabla c_A=(\zeta \circ d_\Gamma)\nabla c^{in}_0 +\mathbf{R} =\;(\zeta \circ d_\Gamma)\theta_{0}^{\prime}(\rho)\left(\frac{\mathbf{n}}{\varepsilon}-\nabla^{\Gamma} h_{\varepsilon}(r,s,t)\right)+\mathbf{R}.
\end{split}
\end{eqnarray}
we utilize Lemma \ref{corollary:corB.3} to obtain that
$\mathbf{R}=\nabla (\zeta \circ d_\Gamma)(c^{in}-\chi_+ +\chi_-)
+(\zeta \circ d_\Gamma)(\varepsilon^2 \nabla c^{in}_2+\varepsilon^3 \nabla c^{in}_3)$ and satisfies
\begin{align} \label{9.1.1}
\|\mathbf{R}\|_{L^\infty(0,T;L^2(\Omega))}\le C\varepsilon^\frac{3}{2}\quad \text{ and } \quad \|\mathbf{R}\|_{ L^\infty(\Omega\times(0,T))}\le C\varepsilon.
\end{align}
It follows from (\ref{2.1}), (\ref{2.2}), (\ref{4.1.1}) and (\ref{2.25.1})  that
\begin{eqnarray}
\begin{split} \label{2.3}
&\partial_t\mathbf{w}+\mathbf{w}\nabla \mathbf{{\tilde v}}_A+\mathbf{v}_\varepsilon \nabla \mathbf{w}-\Delta \mathbf{w}+\nabla (p_\varepsilon-p_A) \\
=&\;-(\zeta \circ d_\Gamma)\Theta_1 -\varepsilon \,\operatorname{div}(\nabla u\otimes \nabla u)-\varepsilon (\zeta \circ d_\Gamma)\,\operatorname{div}(\nabla c_0^{in}\otimes \nabla u+\nabla u\otimes \nabla c_0^{in}) \\
&-\varepsilon \,\nabla(\zeta \circ d_\Gamma)(\nabla c_0^{in}\otimes \nabla u+\nabla u\otimes \nabla c_0^{in})-[\varepsilon \operatorname{div} (\mathbf{R}\otimes \nabla u+\nabla u\otimes \mathbf{R}+\Theta_4)+\Theta_3]  \\
&-\varepsilon^3 \operatorname{div}[(\zeta \circ d_\Gamma)(\nabla c_0^{in}\otimes \nabla g+\nabla g \otimes \nabla c_0^{in})].
\end{split}
\end{eqnarray}
The Gagliardo-Nirenberg inequality in Lemma \ref{GN} implies that
\begin{align}
\|\nabla u\|_{L^{4}(\Omega \times(0, t))} \leq C\Big(\|\nabla u\|_{L^{\infty}(0, t ; L^{2}(\Omega))}^{\frac{1}{2}}\|\Delta u\|_{L^{2}(\Omega \times(0, t))}^{\frac{1}{2}}+T_{0}^{\frac{1}{4}}\|\nabla u\|_{L^{\infty}(0, t ; L^{2}(\Omega))}\Big),\label{4.11.1}\\
\|\nabla u\|_{L^2(0,t;L^4(\Gamma_t(2\delta))} \le C\|\nabla u\|_{L^2(\Omega\times(0,t))}^{\frac{1}{2}}\|\Delta u\|_{L^2(\Omega\times(0,t))}^{\frac{1}{2}}+C\|\nabla u\|_{L^2(\Omega\times(0,t))} \notag
\end{align}
for any $t \in (0,T_0]$.

Thus, we can apply the energy method to \eqref{2.3}, namely, we first multiply (\ref{2.3}) by $\mathbf{w}$ and
integrate the resulting equality over $\Omega \times (0,t)$; and then we have to estimate term by term.
Notice that the capillary term $(\nabla u\otimes \nabla u) \nabla \mathbf{w}$ can be bounded by employing a similar argument to that
in the proof of \cite[ Theorem 4.1]{Fei}, while the remaining terms can be handled in a similar way to that used in the proof
of Theorem 3.5 and Proposition 3.6 in \cite{Liu Y N}. Based on the fact that $\nabla u=\nabla_\tau u+\mathbf{n}$ and
$\partial_\mathbf{n} u$, $\mathbf{n} \otimes \mathbf{n}: \nabla  \mathbf{w}=\partial_{\mathbf{n}} \mathbf{w}_{\mathbf{n}}
=-\operatorname{div}_{\tau} \mathbf{w}$,
we arrive at (\ref{9.1.4}) by combining Proposition \ref{Proposition:prop1.4} with (\ref{9.6.2}), (\ref{4.1.1})-(\ref{9.10.1}),
and (\ref{9.1.1}). The details will be omitted for simplicity of presentation.
\end{proof}

\subsection {Estimates of the error equation of the order parameter $c_{\varepsilon}$}
Let $\mathbf{\tilde{v}}_A$ be defined by (\ref{9.27.1}) and the assumptions (\ref{9.6.2}) be satisfied.
Based on Theorem 4.5 and the proof of \cite[Theorem 1.3]{Liu Y N}, we are able to obtain
\begin{align} \label{12.11.1}
\partial_t \, c_A+\mathbf{\tilde{v}}_A \cdot \nabla c_A-\Delta c_A+\frac{1}{\varepsilon^2}f'(c_A)+\mathbf{w}|_\Gamma \, \nabla c_{A,0}=\mathcal{C}+\mathbf{w}|_\Gamma \, \mathbf{Q},
\end{align}
where $c_{A,0}=\zeta \circ d_{\Gamma} c^{i n}_0(x, t)+(1-\zeta \circ d_{\Gamma})(c_{+}^{o u t} \chi_{+}+c_{-}^{o u t} \chi_{-})$.
Moreover, the following desired estimates hold.
\begin{align} \label{4.12.1}
\|\mathcal{C}\|_{L^{2}\left(\Gamma_{t}(2 \delta)\times(0,T)\right)} \leq C \varepsilon^{\frac{5}{2}} \quad \; \text{ and } \quad\;  \|\mathbf{Q}\|_{L^\infty(0,T;L^2(\Gamma_{t}(2 \delta)))}\le C \varepsilon^\frac{5}{2},
\end{align}
where $C$ is a positive constant depending only on $M$.

Next, we come to estimate the error function of the order parameter $c_\varepsilon$. From (\ref{10.19.1}) and (\ref{12.11.1}) we get
\begin{eqnarray} \label{eq:3.1}
\partial_tu+\mathbf{w}\nabla c_A-\mathbf{w}|_\Gamma \, \nabla c_{A,0}+\mathbf{v}_\varepsilon\nabla u-\Delta u+\frac{1}{\varepsilon^2}[f'(c_\varepsilon)-f'(c_A)]=-\mathcal{C}-\mathbf{w}|_\Gamma \, \mathbf{Q}.
\end{eqnarray}

In this subsection, the main task is to prove the following proposition.
\begin{prop}
\label{P3.2}
Under the a priori assumptions $(\ref{9.6.2})$, there exists a generic constant $C(M)$, such that for any $t \in (0,T)$,
the solution to (\ref{eq:3.1}) satisfies the following estimate:
\begin{eqnarray}
\begin{split}  \label{9.1.5}
&\frac{1}{2\varepsilon}\, \|u(t)\|_{L^2(\Omega)}^2+\frac{3}{4\varepsilon}\|\nabla_\tau u\|_{L^2(\Omega\times(0,t))}^2 +\frac{3}{4}\varepsilon \|\nabla u\|^2_{L^2(\Omega\times(0,t))}\\
\le &\; \frac{1}{2\varepsilon}\|u(0)\|_{L^2(\Omega)}^2+C \|\mathbf{w}\|_{L^2(\Omega\times(0,t))}^2+ \eta\|\nabla\mathbf{w}\|_{L^2(\Omega\times(0,t))}^2 +C \frac{\|u\|_{L^2(\Omega\times(0,t))}^2}{\varepsilon}\\
&+C\,T^\frac{1}{2}\frac{\|u\|_{L^\infty(0,t;L^2(\Omega))}^4}{\varepsilon^6}+ C\varepsilon^4,
\end{split}
\end{eqnarray}
for a suitably small constant $\eta$.
\end{prop}

\begin{proof}
To prove this proposition, it is convenient to rewrite (\ref{eq:3.1}) as the following form:
\begin{eqnarray}
\begin{split} \label{4.11.2}
&\partial_t u+\mathbf{v}_\varepsilon\nabla u-\Delta u+\frac{1}{\varepsilon^2}f''(c_A)\,u\\
=&-\frac{1}{\varepsilon^2}[f'(c_\varepsilon)-f'(c_A)-f''(c_A)\,u]-\mathcal{C}-[\mathbf{w} \nabla c_A-\mathbf{w}|_\Gamma \, \nabla c_{A,0}]-\mathbf{w}|_\Gamma \, \mathbf{Q}.
\end{split}
\end{eqnarray}
  Now, we multiply (\ref{4.11.2}) by $u/\varepsilon$ in $L^2$ and integrate by parts to get
\begin{align*}
& \frac{1}{2\varepsilon} \|u(t)\|_{L^2(\Omega)}^2-\frac{1}{2\varepsilon} \|u(0)\|_{L^2(\Omega)}^2+\int_0 ^t \!\int_\Omega \mathbf{v}_\varepsilon \nabla(\frac{1}{2\varepsilon}u^2)\,\mathrm{d}x\,\mathrm{d}\varsigma+\frac{1}{\varepsilon}\int_0 ^t \!\int_\Omega |\nabla u|^2+\frac{1}{\varepsilon^2}f''(c_A)\,u^2\,\mathrm{d}x\,\mathrm{d}\varsigma\\
=&\int_0 ^t \!\int_\Omega \{ -\frac{1}{\varepsilon^2}[f'(c_\varepsilon)-f'(c_A)-f''(c_A)\,u]-\mathcal{C}-[\mathbf{w} \nabla c_A-\mathbf{w}|_\Gamma \, \nabla c_{A,0}]-\mathbf{w}|_\Gamma \, \mathbf{Q}\} \times \frac{u}{\varepsilon} \,\mathrm{d}x\,\mathrm{d}\varsigma.
\end{align*}
By the following decomposition:
\begin{align*}
&\frac{1}{\varepsilon}\int_0 ^t \!\int_\Omega|\nabla u|^2+\frac{1}{\varepsilon^2} f''(c_A) u^2 \,\mathrm{d}x\,\mathrm{d}\varsigma\\
=&\frac{1-\varepsilon^2}{\varepsilon}\int_0 ^t \!\int_\Omega|\nabla u|^2+\frac{1}{\varepsilon^2} f''(c_A) u^2 \,\mathrm{d}x\,\mathrm{d}\varsigma+\varepsilon \int_0 ^t \!\int_\Omega|\nabla u|^2+\frac{1}{\varepsilon^2} f''(c_A) u^2 \,\mathrm{d}x\,\mathrm{d}\varsigma,
\end{align*}
where the second term on the right hand side will be used to cancel the term of $\eta \varepsilon\|\nabla u\|_{L^2(\Omega\times(0,t))}^2$ in (\ref{9.1.4}).

Recalling (\ref{2.2}) and the definition of $c_{A,0}$, we see that to control $\nabla c_{A,0}$,
it suffices to estimate $\nabla c_{A,0}$ for the case $|d|< 2\delta$. For $|d|< 2\delta$, we have
\begin{align}
\nabla c_A&=\nabla c_{A,0} +\nabla [(\zeta \circ d_\Gamma)(\varepsilon^2 c^{in}_2+\varepsilon^3 c^{in}_3)]\triangleq\;\nabla c_{A,0} +\tilde{\mathbf{Q}},\notag\\
\nabla c_{A,0}&\triangleq\;(\zeta \circ d_\Gamma)\nabla c^{in}_0 +\mathbf{Q}=\;(\zeta \circ d_\Gamma)\theta_{0}^{\prime}(\rho)\left(\frac{\mathbf{n}}{\varepsilon}-\nabla^{\Gamma} h_{\varepsilon}(r,s,t)\right)+\mathbf{Q}\label{12.21.3},
\end{align}
where $\mathbf{Q}$ and $\tilde{\mathbf{Q}}$ satisfy
\begin{align} \label{12.21.2}
\|(\mathbf{Q},\varepsilon \tilde{\mathbf{Q}})\|_{L^\infty(0,T;L^2(\Gamma_{t}(2 \delta)))}\le C \varepsilon^\frac{5}{2} \quad\text{and}
\quad \|(\mathbf{Q},\varepsilon \tilde{\mathbf{Q}})\|_{ L^\infty(\Gamma_{t}(2 \delta)\times(0,T))}\le C\varepsilon^2
\end{align}
due to Lemma \ref{corollary:corB.3}. By Lemma 3.9 in \cite{Abels 4}, we obtain
$$
\|u\|_{L^{3}\left(\Gamma_{t}\left(2\delta\right)\right)}^{3} \leq  C \|(\nabla_\tau u, u)\|_{L^2(\Gamma_t(2\delta))}^\frac{1}{2} \,\|(\partial_\mathbf{n} u, u)\|_{L^2(\Gamma_t(2\delta))}^\frac{1}{2} \,\|u\|_{L^2(\Gamma_t(2\delta))}^2.
$$
Furthermore, according to the divergence-free condition, we have
$\partial_{\mathbf{n}} \mathbf{w}_{\mathbf{n}}+\operatorname{div}_{\tau} \mathbf{w}=\operatorname{div} \mathbf{w}=0$.

 Keeping in mind that
$\operatorname{div}\mathbf{v}_\varepsilon=0$, and employing the trace theorem, Proposition \ref{theorem:thm2.10}, and
(\ref{4.12.1}), (\ref{4.11.2}), (\ref{12.21.3}) and (\ref{12.21.2}), we can adopt calculations similar to those used
in the proof of \cite[Lemmas 5.1 and 5.3]{Liu Y N} to deduce (\ref{9.1.5}).
This completes the proof.
\end{proof}

\subsection{Estimate of derivatives of solutions to the error equations}
As aforementioned, to handle the capillary term $\operatorname{div}(\nabla c_{\varepsilon} \otimes \nabla c_{\varepsilon})$,
it suffices to derive the estimates of the derivatives.
However, there is no desired spectral estimate as in Proposition \ref{theorem:thm2.10} for the estimates of the derivatives, we have to
to deal with the singular term of $\frac{1}{\varepsilon^{2}}[f'(c_\varepsilon)-f'(c_A)]$ directly.
To this end, Abels and Fei multiplied the equation of the error function for $u$ by $\Delta u$ in \cite{Fei}.
Instead in this paper, we come up with the new multiplier
$\mu_{\varepsilon}-\mu_A=-\varepsilon \Delta c_{\varepsilon}+\frac{1}{\varepsilon} f^{\prime}(c_{\varepsilon})
+\varepsilon \Delta c_{A}-\frac{1}{\varepsilon} f^{\prime}(c_{A})$
for estimating the error of the derivatives. This leads to estimating the term
$\frac{1}{\varepsilon^{2}}[f'(c_\varepsilon)-f'(c_A)] \partial_t u$. Intuitively, the second order derivative of $\Delta u$ leads
to a singular factor of $\varepsilon^{-2}$, while the first order derivative of $\partial_t u$ only produces $\varepsilon^{-1}$,
which will improve the estimate of the error function $\nabla (c_{\varepsilon}-c_A)$ by $\varepsilon^{\frac{1}{2}}$.
This is one of the key observations in this paper.

For this purpose, the error equation (\ref{eq:3.1}) is rewritten as
\begin{align} \label{2.26.2}
\partial_t u+(\mathbf{v}_\varepsilon \nabla c_{\varepsilon}-\mathbf{\tilde{v}}_A \nabla c_A)+\frac{\mu_{\varepsilon}-\mu_A}{\varepsilon}=\mathbf{w}|_\Gamma \, \nabla c_{A,0}-\mathbf{w}|_\Gamma \, \mathbf{Q}-\mathcal{C}.
\end{align}
It suffices to prove the following proposition.
\begin{prop}
\label{P3.3}
Let $(\mathbf{w},u)$ be defined by (\ref{9.1.3}) and the a priori assumptions $(\ref{9.6.2})$ be satisfied.
Then there is a generic constant $C(M)>0$ independent of $\varepsilon$, such that for any $t \in (0,T)$, the following estimate holds.
\begin{eqnarray}
\begin{split}     \label{9.1.6}
&\frac{\varepsilon^2}{2}\, \|\nabla u(t)\|_{L^2(\Omega)}^2+\frac{3}{4}\| \mu_{\varepsilon}-\mu_A\|^2_{L^2(\Omega\times(0,t))} +\frac{3}{4}\varepsilon^3 \| \Delta u\|^2 _{L^2(\Omega\times(0,t))} \\
&+\frac{1}{8}\| u(t)\|_{L^4(\Omega)}^4-\frac{1}{4}\| u(t)\|_{L^2(\Omega)}^2+\frac{3}{4} \| (c_A u)(t)\|_{L^2(\Omega)}^2+\frac{1}{\varepsilon}\| f'(c_\varepsilon)-f'(c_A)\|^2_{L^2(\Omega\times(0,t))} \\
\le&\; \frac{\varepsilon^2}{2} \|\nabla u(0)\|_{L^2(\Omega)}^2+\frac{1}{8}\| u(0)\|_{L^4(\Omega)}^4-\frac{1}{4}\| u(0)\|_{L^2(\Omega)}^2+\frac{3}{4} \| (c_A u)(0)\|_{L^2(\Omega)}^2+\frac{1}{2} \| (c_A u^3)(0)\|_{L^1(\Omega)}\\
&+\;C \|\mathbf{w}\|_{L^2(\Omega\times(0,t))}^2+\eta \|\nabla\mathbf{w}\|_{L^2(\Omega\times(0,t))}^2\\
&+C(\varepsilon^4\|\nabla u\|_{L^{\infty}(0, t ; L^{2}(\Omega))}^2 \|\Delta u\|_{L^{2}(\Omega \times(0, t))}^2+T_{0}\varepsilon^4\|\nabla u\|_{L^{\infty}(0, t ; L^{2}(\Omega))}^4\\
&\quad\quad+\varepsilon^2\|\nabla u\|_{L^{\infty}(0, t ; L^{2}(\Omega))} \|\Delta u\|_{L^{2}(\Omega \times(0, t))} +T_{0}\varepsilon^2\|\nabla u\|_{L^{\infty}(0, t ; L^{2}(\Omega))}^2)\|\mathbf{w}\|_{L^{\infty}(0, t ; L^{2}(\Omega))}^2\\
&+C\varepsilon (\|u\|_{L^{\infty}(0, t ; L^{2}(\Omega))}+\|\nabla u\|_{L^{\infty}(0, t ; L^{2}(\Omega))})\frac{\|u\|_{L^{\infty}(0, t ; L^{2}(\Omega))}^2}{\varepsilon}\\
&+C( 1+\|u\|_{L^\infty(0,t;L^2(\Omega))}+\frac{\|u\|_{L^\infty(0,t;L^2(\Omega))}^2}{\varepsilon^2})\frac{\|u\|_{L^2(\Omega\times(0,t))}^2}{\varepsilon}
+\eta \varepsilon \|\nabla u\|^2_{L^2(\Omega\times(0,t))}+C\varepsilon^7
\end{split}
\end{eqnarray}
for a suitably small $\eta>0$.
\end{prop}
\begin{proof}
It is easy to check that
\begin{align*}
&\int_\Omega (\mu_{\varepsilon}-\mu_A)^2 \,\mathrm{d}x=(1-\varepsilon)\int_\Omega (\mu_{\varepsilon}-\mu_A)^2 \,
\mathrm{d}x+\varepsilon\int_\Omega (\mu_{\varepsilon}-\mu_A)^2 \,\mathrm{d}x \\
=&(1-\varepsilon)\| \mu_{\varepsilon}-\mu_A\|^2_{L^2(\Omega)} +\varepsilon^3 \| \Delta u\|^2 _{L^2(\Omega)}
+\frac{1}{\varepsilon}\| f'(c_\varepsilon)-f'(c_A)\|^2_{L^2(\Omega)} -2\varepsilon \int_\Omega [f'(c_\varepsilon)-f'(c_A)]
\Delta u\,\mathrm{d}x.
\end{align*}
Then, we multiply the equation (\ref{2.26.2}) by $\varepsilon (\mu_{\varepsilon}-\mu_A)$ in $L^2$ and integrate by parts to infer that
\begin{eqnarray}
\begin{split} \label{12.15.1}
&\varepsilon \int_\Omega \partial_t u (\mu_{\varepsilon}-\mu_A)\,\mathrm{d}x+(1-\varepsilon)\| \mu_{\varepsilon}-\mu_A\|^2_{L^2(\Omega)}
+\varepsilon^3 \| \Delta u\|^2 _{L^2(\Omega)} +\frac{1}{\varepsilon}\| f'(c_\varepsilon)-f'(c_A)\|^2_{L^2(\Omega)} \\
=&\,2\varepsilon \int_\Omega [f'(c_\varepsilon)-f'(c_A)]\Delta u\,\mathrm{d}x-\varepsilon \int_\Omega (\mathbf{v}_\varepsilon
\nabla c_{\varepsilon}-\mathbf{\tilde{v}}_A \nabla c_A) (\mu_{\varepsilon}-\mu_A)\,\mathrm{d}x\\
&+\varepsilon\int_\Omega[\mathbf{w}|_\Gamma\nabla c_{A,0}-\mathbf{w}|_\Gamma \, \mathbf{Q}]\,(\mu_{\varepsilon}-\mu_A) \,\mathrm{d}x
-\varepsilon \int_\Omega\mathcal{C}(\mu_{\varepsilon}-\mu_A) \,\mathrm{d}x.
\end{split}
\end{eqnarray}
Since $\mu_{\varepsilon}-\mu_A=-\varepsilon\Delta u+\frac{1}{\varepsilon}[f'(c_\varepsilon)-f'(c_A)]$ and $f'(s)=\frac{s^3-s}{2}$, it follows that
\begin{align*}
&\varepsilon\int_0 ^t \!\int_\Omega \partial_t u (\mu_{\varepsilon}-\mu_A)\,\mathrm{d}x\,\mathrm{d}\varsigma\\
=&\frac{\varepsilon^2}{2}\, \|\nabla u(t)\|_{L^2(\Omega)}^2-\frac{\varepsilon^2}{2} \|\nabla u(0)\|_{L^2(\Omega)}^2+ \int_0 ^t \!\int_\Omega \partial_t u [\frac{1}{2}(u^3-u)+\frac{3}{2}(c_A^2 u+c_A u^2)]\,\mathrm{d}x\,\mathrm{d}\varsigma\\
=&\frac{\varepsilon^2}{2}\, \|\nabla u(t)\|_{L^2(\Omega)}^2+\frac{1}{8}\| u(t)\|_{L^4(\Omega)}^4-\frac{1}{4}\| u(t)\|_{L^2(\Omega)}^2-\frac{\varepsilon^2}{2} \|\nabla u(0)\|_{L^2(\Omega)}^2-\frac{1}{8}\| u(0)\|_{L^4(\Omega)}^4\\
&+\frac{1}{4}\| u(0)\|_{L^2(\Omega)}^2 + \int_0 ^t \!\int_\Omega \frac{3}{2}\partial_t u (c_A^2 u+c_A u^2)\,\mathrm{d}x\,\mathrm{d}\varsigma.
\end{align*}
To deal with $\int_0 ^t \!\int_\Omega \frac{3}{2}\partial_t u (c_A^2 u+c_A u^2)\,\mathrm{d}x\,\mathrm{d}\varsigma$,
we integrate by parts to find that
\begin{align*}
&\int_0 ^t \!\int_\Omega \frac{3}{2}\partial_t u (c_A^2 u+c_A u^2)\,\mathrm{d}x\,\mathrm{d}\varsigma=\frac{3}{2}\int_0 ^t \!\int_\Omega \frac{1}{2}\partial_t(c_A^2 u^2)+\frac{1}{3}\partial_t(c_A u^3) -\partial_t c_A(c_A u^2+\frac{1}{3}u^3) \,\mathrm{d}x\,\mathrm{d}\varsigma\\
&=\frac{3}{4} \int_\Omega c_A^2 u^2 \,\mathrm{d}x-\frac{3}{4} \int_\Omega c_A(0)^2 u(0)^2 \,\mathrm{d}x+\frac{3}{2}\int_0 ^t \!\int_\Omega \frac{1}{3}\partial_t(c_A u^3) -\partial_t c_A(c_A u^2+\frac{1}{3}u^3) \,\mathrm{d}x\,\mathrm{d}\varsigma.
\end{align*}
Consequently, integrating (\ref{12.15.1}) over $[0, t]$, we arrive at
\begin{align*}
&\frac{\varepsilon^2}{2}\, \|\nabla u(t)\|_{L^2(\Omega)}^2+(1-\varepsilon)\| \mu_{\varepsilon}-\mu_A\|^2_{L^2(\Omega\times(0,t))} +\varepsilon^3 \| \Delta u\|^2 _{L^2(\Omega\times(0,t))} \\
&+\frac{1}{8}\| u(t)\|_{L^4(\Omega)}^4-\frac{1}{4}\| u(t)\|_{L^2(\Omega)}^2+\frac{3}{4} \| (c_A u)(t)\|_{L^2(\Omega)}^2+\frac{1}{\varepsilon}\| f'(c_\varepsilon)-f'(c_A)\|^2_{L^2(\Omega\times(0,t))} \\
= &\frac{\varepsilon^2}{2} \|\nabla u(0)\|_{L^2(\Omega)}^2+\frac{1}{8}\| u(0)\|_{L^4(\Omega)}^4-\frac{1}{4}\| u(0)\|_{L^2(\Omega)}^2+\frac{3}{4} \| (c_A u)(0)\|_{L^2(\Omega)}^2\\
&- \frac{3}{2}\int_0 ^t \!\int_\Omega \frac{1}{3}\partial_t(c_A u^3) -\partial_t c_A(c_A u^2+\frac{1}{3}u^3) \,\mathrm{d}x\,\mathrm{d}\varsigma +2\varepsilon \int_0 ^t \!\int_\Omega [f'(c_\varepsilon)-f'(c_A)]\Delta u\,\mathrm{d}x\,\mathrm{d}\varsigma\\
&-\varepsilon \int_0 ^t \!\int_\Omega (\mathbf{v}_\varepsilon \nabla c_{\varepsilon}-\mathbf{\tilde{v}}_A \nabla c_A) (\mu_{\varepsilon}-\mu_A)\,\mathrm{d}x\,\mathrm{d}\varsigma+\varepsilon \int_0 ^t \!\int_\Omega[\mathbf{w}|_\Gamma\nabla c_{A,0}-\mathbf{w}|_\Gamma \, \mathbf{Q}]\,(\mu_{\varepsilon}-\mu_A) \,\mathrm{d}x\,\mathrm{d}\varsigma\\
&-\varepsilon \int_0 ^t \!\int_\Omega\mathcal{C}(\mu_{\varepsilon}-\mu_A) \,\mathrm{d}x\,\mathrm{d}\varsigma\\
\triangleq&\; \frac{\varepsilon^2}{2} \|\nabla u(0)\|_{L^2(\Omega)}^2+\frac{1}{8}\| u(0)\|_{L^4(\Omega)}^4-\frac{1}{4}\| u(0)\|_{L^2(\Omega)}^2+\frac{3}{4} \| (c_A u)(0)\|_{L^2(\Omega)}^2+\sum_{k=1}^5 J_k.
\end{align*}
Here the terms $J_k$ ($k=1,\cdots ,5$) will be bounded below.

$\bullet$ For $J_2$ and $J_5$, we use H\"{o}lder's inequality and (\ref{4.12.1}) to have
\begin{align}
|J_2|=&\Big|2\varepsilon \int_0 ^t \!\int_\Omega [f'(c_\varepsilon)-f'(c_A)]\Delta u \,\mathrm{d}x\,\mathrm{d}\varsigma \Big|
\leq C\varepsilon \| f''(c_\theta)\|_{L^\infty(\Omega\times(0,t))} \| u\|_{L^2(\Omega\times(0,t))} \|\Delta u\|_{L^2(\Omega\times(0,t))}\notag\\
\leq & C\frac{\|u\|_{L^2(\Omega\times(0,t))}^2}{\varepsilon}+\eta \, \varepsilon^3\|\Delta u\|_{L^2(\Omega\times(0,t))}^2, \label{3.14}
\end{align}
and
\begin{align}
|J_5|=&\Big|-\varepsilon \! \int_0 ^t \!\int_\Omega\mathcal{C}(\mu_{\varepsilon}-\mu_A) \,\mathrm{d}x\,\mathrm{d}\varsigma \Big| \leq C\varepsilon^{\frac{7}{2}} \|(\mu_{\varepsilon}-\mu_A)\|_{L^2(\Omega \times (0,t))} \nonumber\\
\le& C\varepsilon^7 +\eta \|(\mu_{\varepsilon}-\mu_A)\|_{L^2(\Omega\times(0,t))}^2. \label{3.13}
\end{align}
$\bullet$ To control $J_1$, we take into account that
\begin{align*}
\partial_t c_A=\partial_t (\zeta \circ d_\Gamma)(c^{in}-\chi_+ +\chi_-)+(\zeta \circ d_\Gamma)\theta_{0}^{\prime}(\rho)(-\frac{V}{\varepsilon}-\partial_{t}^{\Gamma} h_{\varepsilon})+(\zeta \circ d_\Gamma)(\varepsilon^2
\partial_t c^{in}_2+\varepsilon^3 \partial_t c^{in}_3)
\end{align*}
to conclude $\|\partial_t c_A\|_{L^\infty(\Omega\times(0,t))} \leq C \varepsilon^{-1}$. Consequently,
\begin{align*}
|J_1| =&\Big|-\frac{3}{2}\int_0 ^t \!\int_\Omega \frac{1}{3}\partial_t(c_A u^3) -\partial_t c_A(c_A u^2+\frac{1}{3}u^3) \,\mathrm{d}x\,\mathrm{d}\varsigma\Big|  \\
\le & \frac{1}{2} \int_\Omega c_A(0) u(0)^3 \,\mathrm{d}x +C\|u\|_{L^{\infty}(0,t; L^{2}(\Omega))}\|u\|_{L^{\infty}(0,t; L^{4}(\Omega))}^2\\
& +C\varepsilon^{-1}(\|u\|_{L^2(\Omega\times(0,t))}^2+\|u\|_{L^{\infty}(0, t ; L^{2}(\Omega))}\|u\|_{L^{2}(0, t ; L^{4}(\Omega))}^2).
\end{align*}
Noticing that $\|u\|_{L^4(\Omega)}\le C\|u\|_{L^2(\Omega)}^{\frac{1}{2}}\|\nabla u\|_{L^2(\Omega)}^{\frac{1}{2}}+C\|u\|_{L^2(\Omega)}$,
we obtain
\begin{align}
|J_1|&\leq \frac{1}{2} \int_\Omega c_A(0) u(0)^3 \,\mathrm{d}x+C(\|\nabla u\|_{L^{\infty}(0, t ; L^{2}(\Omega))} \|u\|_{L^{\infty}(0, t ; L^{2}(\Omega))}^2+ \|u\|_{L^{\infty}(0, t ; L^{2}(\Omega))}^3)\notag\\
&+C\varepsilon^{-1}(\|u\|_{L^2(\Omega\times(0,t))}^2+ \|u\|_{L^\infty(0,t;L^2(\Omega))}\|u\|_{L^2(\Omega\times(0,t))}\|\nabla u\|_{L^2(\Omega\times(0,t))}+ \|u\|_{L^\infty(0,t;L^2(\Omega))}\|u\|_{L^2(\Omega\times(0,t))}^2)\notag\\
&\leq \frac{1}{2} \int_\Omega c_A(0) u(0)^3 \,\mathrm{d}x+C(\|u\|_{L^{\infty}(0, t ; L^{2}(\Omega))}+\|\nabla u\|_{L^{\infty}(0, t ; L^{2}(\Omega))})\|u\|_{L^{\infty}(0, t ; L^{2}(\Omega))}^2\notag\\
&\quad +C( 1+\|u\|_{L^\infty(0,t;L^2(\Omega))}+\frac{\|u\|_{L^\infty(0,t;L^2(\Omega))}^2}{\varepsilon^2})\frac{\|u\|_{L^2(\Omega\times(0,t))}^2}{\varepsilon}+\eta \varepsilon \|\nabla u\|^2_{L^2(\Omega\times(0,t))}\label{3.16}.
\end{align}
$\bullet$ The term $J_3$ can be bounded as follows.
\begin{align*}
|J_3| & = \Big|-\varepsilon \int_0 ^t \!\int_\Omega (\mathbf{\tilde{v}}_A \nabla u+\mathbf{w} \nabla c_A+\mathbf{w} \nabla u) (\mu_{\varepsilon}-\mu_A)\,\mathrm{d}x\,\mathrm{d}\varsigma \Big|\\
&\le C\varepsilon (\|\mathbf{\tilde{v}}_A\|_{L^\infty(\Omega\times(0,t))} \|\nabla u\|_{L^2(\Omega\times(0,t))}+\|\mathbf{w}\|_{L^2(\Omega\times(0,t))} \|\nabla c_A\|_{L^\infty(\Omega\times(0,t))}\\
&\quad+\|\mathbf{w}\|_{L^4(\Omega\times(0,t))} \|\nabla u\|_{L^4(\Omega\times(0,t))}) \|\mu_{\varepsilon}-\mu_A\|_{L^2(\Omega\times(0,t))}\\
&\le C\varepsilon (\|\nabla u\|_{L^2(\Omega\times(0,t))}+\varepsilon ^{-1}\|\mathbf{w}\|_{L^2(\Omega\times(0,t))}
+\|\nabla u\|_{L^{4}(\Omega \times(0, t))} \|\mathbf{w}\|_{L^{\infty}(0, t ; L^{2}(\Omega))}^{\frac{1}{2}}
\|\nabla \mathbf{w}\|_{L^{2} (\Omega \times(0, t))}^{\frac{1}{2}}\\
&\quad+\|\nabla u\|_{L^{4}(\Omega \times(0, t))}T_{0}^{\frac{1}{4}}\|\mathbf{w}\|_{L^{\infty}(0, t ; L^{2}(\Omega))} )
\|\mu_{\varepsilon}-\mu_A\|_{L^2(\Omega\times(0,t))},
\end{align*}
where in the last equality, the following estimate is used.
$$
\|\mathbf{w}\|_{L^{4}(\Omega \times(0, t))} \le C\left(\|\mathbf{w}\|_{L^{\infty}(0, t ; L^{2}(\Omega))}^{\frac{1}{2}}
\|\nabla \mathbf{w}\|_{L^{2}(\Omega \times(0, t))}^{\frac{1}{2}}+T_{0}^{\frac{1}{4}}\|\mathbf{w}\|_{L^{\infty}(0, t ; L^{2}(\Omega))}\right),
\;\;\; \forall\, t\in (0,T_{0}].
$$
Moreover, thanks to (\ref{4.11.1}), we have
\begin{eqnarray}
\begin{split} \label{3.12}
|J_3| \leq &\;C \|\mathbf{w}\|_{L^2(\Omega\times(0,t))}^2+C(\varepsilon^4 \|\nabla u\|_{L^{4}(\Omega \times(0, t))}^4+T_0^{\frac{1}{2}} \varepsilon^2 \|\nabla u\|_{L^{4}(\Omega \times(0, t))}^2 ) \|\mathbf{w}\|_{L^{\infty}(0, t ; L^{2}(\Omega))}^2\\
&+\eta \|\nabla\mathbf{w}\|_{L^2(\Omega\times(0,t))}^2+C\varepsilon^2 \|\nabla u\|^2_{L^2(\Omega\times(0,t))}+\eta \|\mu_{\varepsilon}-\mu_A\|^2_{L^2(\Omega\times(0,t))}\\
\leq &\;C \|\mathbf{w}\|_{L^2(\Omega\times(0,t))}^2+\eta \|\nabla\mathbf{w}\|_{L^2(\Omega\times(0,t))}^2+C\varepsilon^2 \|\nabla u\|^2_{L^2(\Omega\times(0,t))}+\eta \|\mu_{\varepsilon}-\mu_A\|^2_{L^2(\Omega\times(0,t))}\\
&+C(\varepsilon^4\|\nabla u\|_{L^{\infty}(0, t ; L^{2}(\Omega))}^2 \|\Delta u\|_{L^{2}(\Omega \times(0, t))}^2+T_{0}\varepsilon^4\|\nabla u\|_{L^{\infty}(0, t ; L^{2}(\Omega))}^4\\
&\quad\quad+\varepsilon^2\|\nabla u\|_{L^{\infty}(0, t ; L^{2}(\Omega))} \|\Delta u\|_{L^{2}(\Omega \times(0, t))} +T_{0}\varepsilon^2\|\nabla u\|_{L^{\infty}(0, t ; L^{2}(\Omega))}^2)\|\mathbf{w}\|_{L^{\infty}(0, t ; L^{2}(\Omega))}^2
\end{split}
\end{eqnarray}
for a suitably small $\eta$.\\
$\bullet$ We use (\ref{2.25.1}) and (\ref{12.21.3}) to write $J_4$ as
\begin{align*}
J_4=\int_0 ^t\!\int_{\Gamma_{t}(2 \delta)}(\zeta \circ d_\Gamma)\mathbf{w}|_\Gamma\theta_{0}^{\prime}(\rho)(\mathbf{n}
-\varepsilon \nabla_\tau h_{\varepsilon})(\mu_{\varepsilon}-\mu_A) \,\mathrm{d}x\,\mathrm{d}\varsigma \triangleq\; J_{41}+J_{42},
\end{align*}
where the term $J_{41}$ can be estimated as follows.
\begin{align*}
|J_{41}|&=\Big|\int_0 ^t\!\int_{\Gamma_{t}(2 \delta)}(\zeta \circ d_\Gamma)\mathbf{w}|_\Gamma\theta_{0}^{\prime}(\rho)\mathbf{n}(\mu_{\varepsilon}-\mu_A) \,\mathrm{d}x\,\mathrm{d}\varsigma\Big|\\
&\le C \varepsilon^{\frac{1}{4}}\|\theta_{0}'\|_{L^\infty(0,t;L^{4}(\mathbb{R}))} \|\mathbf{n} \|_{L^\infty(\Gamma_{t}(2 \delta)\times(0,t))} \|\mathbf{w}|_\Gamma\|_{L^2(0,t;L^{2,4}\left(\Gamma_{t}(2 \delta)\right))}\|\mu_{\varepsilon}-\mu_A\|_{L^2(\Gamma_{t}(2 \delta)\times(0,t))}\\
&\le C \varepsilon^{\frac{1}{4}} \|\nabla \mathbf{w}\|_{L^2(\Omega\times(0,t))}\|\mu_{\varepsilon}-\mu_A\|_{L^2(\Omega\times(0,t))}\le C \varepsilon^{\frac{1}{2}} \|\nabla \mathbf{w}\|_{L^2(\Omega\times(0,t))}^2+\eta \|\mu_{\varepsilon}-\mu_A\|_{L^2(\Omega\times(0,t))}^2.
\end{align*}
As for $J_{42}$, noticing that $h_{\varepsilon}= h_{1}+\varepsilon h_{2}$, one finds that
\begin{align*}
|J_{42}|=&\Big|-\int_0 ^t\!\int_{\Gamma_{t}(2 \delta)}\varepsilon (\zeta \circ d_\Gamma)\mathbf{w}|_\Gamma\theta_{0}^{\prime}(\rho)\nabla_\tau h_{\varepsilon}(\mu_{\varepsilon}-\mu_A) \,\mathrm{d}x\,\mathrm{d}\varsigma\Big|\\
\le &C\varepsilon \|\mathbf{w}|_\Gamma\|_{L^2(\Gamma_{t}(2 \delta)\times(0,t))} \|\theta_{0}' \nabla_{\tau} h_{1} \|_{L^\infty(\Gamma_{t}(2 \delta)\times(0,t))} \|\mu_{\varepsilon}-\mu_A\|_{L^2(\Gamma_{t}(2 \delta)\times(0,t))}\\
& +C\varepsilon^2 \|\mathbf{w}|_\Gamma\|_{L^2(0,t;L^{4,2}\left(\Gamma_{t}(2 \delta)\right))} \|\theta_{0}' \|_{L^\infty(\Gamma_{t}(2 \delta)\times(0,t))}\|\nabla_{\tau} h_{2} \|_{L^\infty(0,t;L^{4}(\mathbb{T}^{1}))} \|\mu_{\varepsilon}-\mu_A\|_{L^2(\Gamma_{t}(2 \delta)\times(0,t))}\\
\le& C\varepsilon^2 \|\nabla \mathbf{w}\|_{L^2(\Omega\times(0,t))}^2+\eta \|\mu_{\varepsilon}-\mu_A\|_{L^2(\Omega\times(0,t))}^2.
\end{align*}
Consequently, putting the above two estimates together, we conclude
\begin{align} \label{12.22.4}
|J_4| \le C \varepsilon^{\frac{1}{2}} \|\nabla \mathbf{w}\|_{L^2(\Omega\times(0,t))}^2+\eta \|\mu_{\varepsilon}
-\mu_A\|_{L^2(\Omega\times(0,t))}^2.
\end{align}

From the estimates (\ref{3.14}), (\ref{3.13}), and (\ref{3.16}), (\ref{3.12}) and (\ref{12.22.4}), we obtain Proposition \ref{P3.3}.
\end{proof}

\section{Proof of Theorem \ref{theorem:thm1.1} }
Based on the prior estimates established in Section 3, we are ready to prove Theorem \ref{theorem:thm1.1} by the a priori energy estimate method.
\begin{proof}
Putting the estimates (\ref{9.1.4}), (\ref{9.1.5}) and (\ref{9.1.6}) together, and choosing $\eta$ suitably small, we obtain
\begin{align*}
&\frac{1}{2}\, \|\mathbf{w}(t)\|_{L^2(\Omega)}^2+(\frac{1}{2\varepsilon}-\frac{1}{4})\|u(t)\|_{L^2(\Omega)}^2+\frac{\varepsilon^2}{2}\, \|\nabla u(t)\|_{L^2(\Omega)}^2\\
&+\frac{1}{2}\|\nabla \mathbf{w}\|_{L^2(\Omega\times(0,t))}^2 +\frac{1}{2\varepsilon}\|\nabla_\tau u\|_{L^2(\Omega\times(0,t))}^2+\frac{\varepsilon}{2} \|\nabla u\|^2_{L^2(\Omega\times(0,t))}+\frac{3}{4}\varepsilon^3\|\Delta u\|_{L^2(\Omega\times(0,t))}^2\\
&+\frac{3}{4}\| \mu_{\varepsilon}-\mu_A\|^2_{L^2(\Omega\times(0,t))} +\frac{1}{8}\| u(t)\|_{L^4(\Omega)}^4+\frac{3}{4} \| (c_A u)(t)\|_{L^2(\Omega)}^2+\frac{1}{\varepsilon}\| f'(c_\varepsilon)-f'(c_A)\|^2_{L^2(\Omega\times(0,t))} \\
\leq& \;\frac{1}{2} \|\mathbf{w}(0)\|_{L^2(\Omega)}^2+(\frac{1}{2\varepsilon}-\frac{1}{4})\|u(0)\|_{L^2(\Omega)}^2+\frac{\varepsilon^2}{2} \|\nabla u(0)\|_{L^2(\Omega)}^2+\frac{1}{8}\| u(0)\|_{L^4(\Omega)}^4+\frac{3}{4} \| (c_A u)(0)\|_{L^2(\Omega)}^2\\
&+\frac{1}{2} \| (c_A u^3)(0)\|_{L^1(\Omega)}+C \|\mathbf{w}\|_{L^2(\Omega\times(0,t))}^2+C\varepsilon \|\nabla u\|_{L^{\infty}(0, t ; L^{2}(\Omega))}^4+C\varepsilon^3 \|\Delta u\|_{L^{2}(\Omega \times(0, t))}^4\\
&+C(\varepsilon^4\|\nabla u\|_{L^{\infty}(0, t ; L^{2}(\Omega))}^2 \|\Delta u\|_{L^{2}(\Omega \times(0, t))}^2+T_{0}\varepsilon^4\|\nabla u\|_{L^{\infty}(0, t ; L^{2}(\Omega))}^4\\
&\quad\quad+\varepsilon^2\|\nabla u\|_{L^{\infty}(0, t ; L^{2}(\Omega))} \|\Delta u\|_{L^{2}(\Omega \times(0, t))} +T_{0}\varepsilon^2\|\nabla u\|_{L^{\infty}(0, t ; L^{2}(\Omega))}^2)\|\mathbf{w}\|_{L^{\infty}(0, t ; L^{2}(\Omega))}^2\\
&+C\varepsilon (\|u\|_{L^{\infty}(0, t ; L^{2}(\Omega))}+\|\nabla u\|_{L^{\infty}(0, t ; L^{2}(\Omega))})\frac{\|u\|_{L^{\infty}(0, t ; L^{2}(\Omega))}^2}{\varepsilon}\\
&+C( 1+\|u\|_{L^\infty(0,t;L^2(\Omega))}+\frac{\|u\|_{L^\infty(0,t;L^2(\Omega))}^2}{\varepsilon^2})\frac{\|u\|_{L^2(\Omega\times(0,t))}^2}{\varepsilon}+C_1\,T^\frac{1}{2}\frac{\|u\|_{L^\infty(0,t;L^2(\Omega))}^4}{\varepsilon^6}+ C_2\,\varepsilon^4.
\end{align*}

In the calculations that follow, we further require the following a priori assumptions:
\begin{align}
\label{4.1}
\|u\|_{L^\infty(0,t;L^2(\Omega))} \le R \,\varepsilon^{\frac{5}{2}},\quad \|\nabla u\|_{L^\infty(0,t;L^2(\Omega))} \le R \,\varepsilon^{\frac{7}{8}}\quad \text{and}\quad \|\Delta u\|_{L^{2}(\Omega \times(0, t))} \le R \,\varepsilon^{\frac{3}{8}}.
\end{align}
Let $C_0=\max\{C,C_1, C_2\} \le \frac{R^2}{100}$ and the initial data satisfy
\begin{align*}
\frac{1}{2} \|\mathbf{w}(0)\|_{L^2(\Omega)}^2 & + (\frac{1}{2\varepsilon}-\frac{1}{4})\|u(0)\|_{L^2(\Omega)}^2+\frac{\varepsilon^2}{2}
\|\nabla u(0)\|_{L^2(\Omega)}^2\\
& +\frac{1}{8}\| u(0)\|_{L^4(\Omega)}^4+\frac{3}{4} \| (c_A u)(0)\|_{L^2(\Omega)}^2 +\frac{1}{2} \| (c_A u^3)(0)\|_{L^1(\Omega)}
\le C_0 \varepsilon^4.
\end{align*}
Then, we have
\begin{align*}
&\frac{1}{4}\, \|\mathbf{w}(t)\|_{L^2(\Omega)}^2+\frac{1}{4\varepsilon}\, \|u(t)\|_{L^2(\Omega)}^2+\frac{\varepsilon^2}{2}\,
\|\nabla u(t)\|_{L^2(\Omega)}^2 \\
&+\frac{1}{2}\|\nabla \mathbf{w}\|_{L^2(\Omega\times(0,t))}^2 +\frac{1}{2\varepsilon}\|\nabla_\tau u\|_{L^2(\Omega\times(0,t))}^2+\frac{\varepsilon}{2} \|\nabla u\|^2_{L^2(\Omega\times(0,t))}+\frac{\varepsilon^3}{2}
\|\Delta u\|_{L^2(\Omega\times(0,t))}^2 \\
&+\frac{1}{2} \| \mu_{\varepsilon}-\mu_A\|^2_{L^2(\Omega\times(0,t))} +\frac{1}{8}\| u(t)\|_{L^4(\Omega)}^4
+\frac{3}{4} \| (c_A u)(t)\|_{L^2(\Omega)}^2+\frac{1}{\varepsilon}\| f'(c_\varepsilon)-f'(c_A)\|^2_{L^2(\Omega\times(0,t))} \\
\leq& \;C_0 \|\mathbf{w}\|_{L^2(\Omega\times(0,t))}^2+C_0\frac{\|u\|_{L^2(\Omega\times(0,t))}^2}{\varepsilon}
+ C_0(1+T^\frac{1}{2}R^4)\varepsilon^4.
\end{align*}
Consequently, an application of Gronwall's inequality to the above inequality leads to
\begin{align*}
&\frac{1}{4}\, \|\mathbf{w}(t)\|_{L^2(\Omega)}^2+\frac{1}{4\varepsilon}\, \|u(t)\|_{L^2(\Omega)}^2+\frac{\varepsilon^2}{2}\, \|\nabla u(t)\|_{L^2(\Omega)}^2\\
&+\frac{1}{2}\|\nabla \mathbf{w}\|_{L^2(\Omega\times(0,t))}^2 +\frac{1}{2\varepsilon}\|\nabla_\tau u\|_{L^2(\Omega\times(0,t))}^2+\frac{\varepsilon}{2} \|\nabla u\|^2_{L^2(\Omega\times(0,t))}+\frac{\varepsilon^3}{2}\|\Delta u\|_{L^2(\Omega\times(0,t))}^2\\
&+\frac{1}{2} \| \mu_{\varepsilon}-\mu_A\|^2_{L^2(\Omega\times(0,t))} +\frac{1}{8}\| u(t)\|_{L^4(\Omega)}^4+\frac{3}{4} \| (c_A u)(t)\|_{L^2(\Omega)}^2+\frac{1}{\varepsilon}\| f'(c_\varepsilon)-f'(c_A)\|^2_{L^2(\Omega\times(0,t))} \\
\leq& \;C_0( 1+T^\frac{1}{2}R^4)\varepsilon^4 (1+C_0 te^{C_0 t})  \leq \,\frac{R^2\varepsilon^4}{16}\qquad
\mbox{for all }t \in [0,T],
\end{align*}
provided that $T$ is sufficiently small. Hence,
$$\|u\|_{L^\infty(0,t;L^2(\Omega))} \le \frac{1}{2}R \,\varepsilon^{\frac{5}{2}},\;\;\;
\|\nabla u\|_{L^\infty(0,t;L^2(\Omega))} \le R \,\varepsilon , \;\;\;
\|\Delta u\|_{L^2(\Omega\times(0,t))} \le R \,\varepsilon^{\frac{1}{2}}.
$$
Therefore, the a priori assumptions (\ref{4.1}) are satisfied. Moreover, the a priori assumptions (\ref{9.6.2})
are also valid by virtue of \cite[Lemma 4.2]{Liu Y N}. So, the proof of Theorem \ref{theorem:thm1.1} is complete.
\end{proof}

\vspace{.1in}

\noindent{\bf Acknowledgments:} The third author thanks Yuning Liu for many discussions and pointing him to the references \cite{Fei, Liu-Y}.
The research of F. Xie was partially supported by NSFC
(Grant No.11831003) and Shanghai Science and Technology Innovation Action Plan (Grant No. 20JC1413000),
and the research of S. Jiang by National Key R\&D Program (2020YFA0712200), National Key Project
(GJXM92579), and NSFC (Grant No. 11631008), the Sino-German Science Center (Grant No. GZ 1465)
and the ISF{NSFC joint research program (Grant No. 11761141008).

\end{document}